\documentclass{amsart}

\newtheorem{theorem}{Theorem}[section]

 \newtheorem{corollary}[theorem]{Corollary}
 \newtheorem{proposition}[theorem]{Proposition}

\theoremstyle{definition}
\newtheorem{definition}[theorem]{Definition}

\newtheorem{conjecture}[theorem]{Conjecture}

\theoremstyle{remark}
\newtheorem{remark}[theorem]{Remark}

\numberwithin{equation}{section}

\newcommand{\diver}{{{\mathrm{div}}}}
\newcommand{\grad}{{{\mathrm{grad}}}}

\newcommand{\pa}{\partial/\partial}

\begin{document}

\title{Dirichlet spectrum and Green function}


\author{G. Pacelli Bessa}
\address{Department of Mathematics, Universidade Federal do Cear\'{a},  
60440-900, Fortaleza, Brazil}
\email{bessa@mat.ufc.br}
\thanks{Research partially supported by the Brazilian's funding agencies FUNCAP-CE, PRONEX-CAPES
and CNPq-Brazil grants \# 301581/2013-4}

\author{Vicent Gimeno}
\address{Department of Mathematics, Universitat Jaume I-IMAC,   E-12071, 
Castell\'{o}n, Spain}
\email{gimenov@mat.uji.es}
\thanks{Research partially supported by  the Universitat Jaume I Research Program Project P1-1B2012-18, and DGI-MINECO grant (FEDER) MTM2013-48371-C2-2-PDGI from Ministerio de Ciencia e Inovaci\'{o}n (MCINN), Spain.}

\author{Luq\'{e}sio P. Jorge}
\address{Department of Mathematics, Universidade Federal do Cear\'{a},  
60440-900, Fortaleza, Brazil}
\email{ljorge@mat.ufc.br}
\thanks{Research partially supported by by the Brazilian's funding agencies FUNCAP-CE, PRONEX-CAPES
and CNPq-Brazil grants \# 305778/2012-9}

\subjclass[2010]{Primary 58C40, 35P15, (53A10)}

\keywords{Dirichlet spectrum, radial spectrum, momentum spectrum, Green functions }

\date{\today}

\dedicatory{}

\begin{abstract}
In the first part of this article we obtain an identity relating  the radial spectrum of rotationally invariant geodesic balls and an isoperimetric  quotient $\sum 1/\lambda_{i}^{\rm rad}=\int V(s)/S(s)ds$, Thms \ref{thmMain1-intro} \& \ref{thmMain2}. We also obtain  upper and lower estimates for the series $\sum \lambda_{i}^{-2}(\Omega)$ where $\Omega$ is an extrinsic ball of a proper minimal surface of $\mathbb{R}^{3}$, Thm \ref{thmMark}. In the second part we show that the first eigenvalue of  bounded domains is given by  iteration of the Green operator and taking the limit, $\lambda_{1}(\Omega)=\lim_{k\to \infty} \Vert G^k(f)\Vert_{2}/\Vert G^{k+1}(f)\Vert_{2}$ for any function $f>0$, Thm \ref{Main2}. In the third part we obtain explicitly the $L^{1}(\Omega, \mu)$-momentum spectrum of a bounded domain $\Omega$ in terms of its Green operator, Thm \ref{thm4.2}. In particular, we obtain the first eigenvalue of a weighted bounded domain in terms of the $L^{1}(\Omega, \mu)$-momentum spectrum,  extending the work of Hurtado-Markvorsen-Palmer on the first eigenvalue of rotationally invariant balls, \cite{HMP}. 
\end{abstract}

\maketitle
\section{Introduction}\label{sec:intro}

Let $\Omega \subseteq M$ be a subset of a Riemannian manifold  $(M, ds^2)$. The Laplace operator $\triangle =\diver \circ \grad\colon C^{\infty}_{0}(\Omega) \to   C^{\infty}_{0}(\Omega) $, acting on the space of smooth functions with compact support in $\Omega$,  is  symmetric, negative-definite and densely defined in $L^{2}(\Omega)$, however it is not self-adjoint.    Considering the  Sobolev spaces $W^{1}_{0}(\Omega)$ 
and $W_{0}^{2}(\Omega)$, the former being the closure of 
$C^{\infty}_{0}(\Omega)$ with respect to the norm 
\[\Vert u \Vert^{2}_{W^{1}(\Omega)}\colon = \int_{\Omega}u^{2}d\nu + 
\int_{\Omega}\vert \grad u \vert^{2}d\nu\]
while the latter consists of 
those functions $u\in W^{1}_{0}(\Omega)$ whose weak Laplacian 
$\triangle u $ exists and belongs to $L^{2}(\Omega)$, i.e. 
\[W_{0}^{2}(\Omega)=\{u\in W^{1}_{0}(\Omega)\colon \triangle u 
\in L^{2}(\Omega)\},\] we have that 
$\mathcal{L}=-\triangle \vert_{W_{0}^{2}(\Omega)}$ is  a non-negative self-adjoint extension of $-\triangle\vert_{C_{0}^{\infty}(\Omega)}$, see \cite{grigoryan-book}.  Let us recall that the spectrum of $\mathcal{L}$, denoted by  $\sigma (\Omega)$, is the set of all $\lambda\in  [0,\infty)$ for which $\mathcal{L}-\lambda I$ is not injective or the inverse operator $(\mathcal{L}-\lambda I)^{-1}$ is unbounded, see \cite{davies}.   The  set of all   $\lambda$ for which  $(\mathcal{L}-\lambda I)$ is not injective, (eigenvalues) is  called {\em point spectrum} $\sigma_{p}(\Omega)$. Each eigenvalue $\lambda \in \sigma_{p}(\Omega)$ is associated to a vector space $V_{\lambda}=\{u\in L^{2}(\Omega)\colon \triangle u + \lambda u =0\}$. The   set  of all isolated eigenvalues  of finite multiplicity, i.e., those $\lambda\in \sigma_{p}(\Omega)$ for which there exists $\epsilon >0$ such that $(\lambda - \epsilon, \lambda +\epsilon)\cap \sigma(\Omega)=\{\lambda\}$ and ${\rm dim}(V_{\lambda})<\infty$ is called the {\em discrete spectrum} and it is denoted by $\sigma_{d}(\Omega)$.  The complement of the discrete spectrum is the {\em essential spectrum},   $\sigma_{\mathrm{ess}}(\Omega)=\sigma (M)\setminus \sigma_{d}(\Omega)$.
When  $\Omega$ is  bounded with smooth boundary  $\partial \Omega$, (possibly empty), then  
 the spectrum  of $\mathcal{L}$ is  discrete, i.e. a sequence  of non-negative real numbers
\[0\leq\overline{\lambda}_{1}(\Omega)<\overline{\lambda_2}(\Omega) < \cdots \nearrow\infty,\] where each eigenvalue $\overline{\lambda}_{k}(\Omega)$ is associated to a finite dimensional vector subspace   $V_k=\{\phi\colon \triangle\phi + \overline{\lambda}_{k}\phi=0\}$ such that $L^{2}(\Omega)=\oplus_{k=1}^{\infty}V_{k}$, where $\phi\in  C^{\infty}(\Omega)\cap C^{0}(\overline{\Omega})$ and $\phi\vert \partial \Omega =0$. The eigenvalues $\overline{\lambda}_{k}$ are  sometimes called \textquotedblleft eigenvalues of $\Omega$\textquotedblright and $\phi\in V_{k}$ are called  eigenfunctions associated to $\overline{\lambda_{k}}$,  see \cite{chavel}, \cite{davies} and \cite[Thm.10.3]{grigoryan-book}.


A classical and  important problem in Riemannian geometry  is the study of the  relations between the spectrum $\sigma(\Omega)$ and the  geometry of $\Omega$, see  \cite{berger}, \cite{BGM}. In full  generality, this is a difficult problem. A reasonable problem is, the study of the spectrum of    {\em rotationally invariant}  geodesic balls, i.e.  balls with metrics invariant  by rotations around the center. The spectrum $\sigma (B(o,r))$ of a  {\em rotationally invariant}  geodesic ball $B(o,r)$ can be decomposed as a union of spectra $\sigma^{l}(B(o,r))$ of a  family of operators $L_{l}$ acting on smooth functions  on $[0,r]$ indexed by the eigenvalues of the sphere $\nu_{l}=l(l+n-2),\,l=0,1,\ldots$, this is,   $\sigma (B(o,r))=\cup_{l=0}^{\infty}\sigma^{l}(B(o,r))$, see  \cite[p. 41]{chavel}, \cite[Chapters 7 \& 8]{C-L}. Our first results concern each spectrum $\sigma^{l}(B(o,r))$ separately. More precisely, Theorem \ref{thmMain1-intro} \& \ref{thmMain2}, states  that  \[\sum_{\lambda\in \sigma^{0}(B(o,r)) } \lambda^{-1}=\int_{0}^{r} V(s)/S(s)ds,\]where $V(s)={\rm vol} (B(o,s))$ and $S(s)={\rm vol}(\partial B(o,s))$. And if $B(o,r)\subset \mathbb{R}^{n}$ then 
 \[\sum_{\lambda\in \sigma^{l}(B(o,r)) } \lambda^{-1}=c(l, n)\cdot\int_{0}^{r} V(s)/S(s)ds,\] $0<c(n,l)<1$, for  $l=1,2\ldots$
 We also observe that a lemma due to Cheng-Li-Yau \cite[Lemma 7]{CLY} implies that if ${\rm dim}(V_{k})\geq 2$ then  the origin of $B(o,r)$, belongs to the nodal set of every $\phi \in V_{k}$, see  Theorem \ref{thmNodal}. To close this   part of the article regarding the spectrum of rotationally invariant balls, i.e., Section \eqref{sec2}, we construct  examples of $4$-dimensional non-rotationally invariant geodesic balls $B(o,r)$ with the same spectrum  $\sigma^{0}(B(o,r))=\sigma^{0}(\tilde{B}(o,r))$ as the geodesic ball $\tilde{B}(o,r)$ of the hyperbolic space $\mathbb{H}^{4}(-1)$, see Example \ref{examp1}.

 In Section \ref{sec3}, we consider proper isometric minimal immersions $\varphi \colon M \to \mathbb{R}^{n}$, with  $\varphi(p)=o$ and  extrinsic balls $\Omega_r$ of radius $r$, i.e. the  connected component   $\Omega_r\subset \varphi^{-1}(B(o,r))$ containing $p$.  The main result in this section,  Theorem \ref{thmMark}, is the following estimate \[ A(m)\cdot \frac{r^m}{{\rm vol}(\Omega_r)}r^4 \leq \sum_{\lambda\in \sigma (\Omega_r)} \lambda^{-2} \leq B(m)\cdot\left(\frac{{\rm vol}(\Omega_r)}{r^m}\right)^{4/m}r^4 , \]if $m={\rm dim}(M)=2,3$. 
 
 In Section \ref{sec4}, we consider bounded subsets $\Omega$ with smooth boundary of weighted manifolds $ (M,ds^2, \mu)$, $d\mu = \psi d\nu$, $\psi >0$,  weighted Laplace operator \[\triangle_{\mu}=\frac{1}{\psi} \diver (\psi\, \grad)\] and its  Green operator $G\colon L^{2}(\Omega, \mu)\to L^{2}(\Omega, \mu)$. A particular  case of the main result of this section is that if $f\in L^{2}(\Omega, \mu)$, $f>0$ then \[\lim_{k\to \infty}\frac{\Vert G^{k}(f)\Vert_{L^2}}{\Vert G^{k+1}(f)\Vert_{L^{2}}}=\lambda_{1}(\Omega) \,\,{\rm and}\,\,\lim_{k\to \infty}\frac{ G^{k}(f)}{\Vert G^{k}(f)\Vert_{L^{2}}}\in {\rm Ker}( \triangle_{\mu}+ \lambda_1), \]where $G^{k}=\stackrel{k-times}{\overbrace{G\circ \cdots G}}$. More details see Theorem \ref{Main2}. In Section \ref{sec5} we consider the following hierarchy Dirichlet problem on a bounded open subset with smooth boundary $\Omega\subset M$,
 \begin{equation*}\left\{ \begin{array}{rll}\phi_{0}&=&1\,\,{\rm in}\,\,\Omega\\
 \triangle \phi_k+k\phi_{k-1}&=&0 \,\,{\rm in}\,\,\Omega\\
 \phi_{k}&=&0 \,\,{\rm on}\,\,\partial\Omega.
 \end{array}\right.
 \end{equation*}The set $\{\mathcal{A}_{k}\}_{k=1}^{\infty}$, $\mathcal{A}_{k}=\int_{\Omega}\phi_kd\mu$ is called the $L^{1}(\Omega, \mu)$-momentum spectrum. The momentum spectrum is intertwined  with the spectrum of $\Omega$. In the  main result of this section, Theorem \ref{thm4.2}, we solve the hierarchy Dirichlet problem and  explicitly give the momentum spectrum the in terms of the Green operator. In Section \ref{sec6}, the main result, see Theorem \ref{ThmMainA}, gives necessary and sufficient conditions for the expansion of the Green function in $L^{2}$-sense, \[g(x,y)=\sum_{i=1}^{\infty} \frac{u_i(x)\cdot u_i(y)}{\lambda_{i}(\Omega)}, \]where $\{u_i\}_{i=1}^{\infty}$ is a orthonormal basis of $L^{2}(\Omega, \mu)$ formed by eigenfunctions. This result is the main tool to prove Theorems \ref{thmMain1-intro}, \ref{thmMain2}, \ref{thmMark}. In Section \ref{sec7} we present  the proofs of all results.
 Notation: in this article, the spectrum of $\Omega$ will be written either as a sequence of  eigenvalues with  repetition,    according to their multiplicities,    $\sigma (\Omega) =\{0\leq\lambda_{1}(\Omega) < \lambda_{2}(\Omega) \leq \cdots \}$ or  as  a sequence of eigenvalues   without repetition $\sigma (\Omega)=\{ 0\leq\overline{\lambda}_{1}(\Omega)<\overline{\lambda_2}(\Omega) < \cdots \}.$

\section{Spectrum of rotationally invariant balls}  \label{sec2}
A Riemannian $m$-manifold is a rotationally invariant $m$-manifold,  also  called  {\em model manifold}, with   radial sectional curvature $-G(r)$ along the geodesics issuing from the origin, where $G\colon \mathbb{R}\to \mathbb{R}$ is a smooth even function, is defined as the quotient space
\[
\mathbb{M}^m_{h}=[0,R_h)\times\mathbb{S}^{m-1}/\sim
\] with ($(t, \theta)\sim (s,\beta)\Leftrightarrow t=s=0$ and $\forall \theta, \beta\in \mathbb{S}^{m-1}$ or $ t=s$ and $\theta=\beta$, endowed with the metric $ds^2_{h}(t,\theta)=dt^2+h^2(t)d\theta^2$
where  $h\colon [0, \infty) \to \mathbb{R}$ is the unique solution of the Cauchy problem \begin{eqnarray}\label{ubs2}
\left\{\begin{array}{l}
h''-Gh=0,\\
h(0)=0, h'(0)=1,
\end{array}\right.
\end{eqnarray} and
$R_h$ is the largest positive real number such that $h\vert_{(0,R_{h})}>0$. If $R_h=\infty$, the manifold $\mathbb{M}^{m}_{h}$ is geodesically complete. The geodesic ball $B(o,r)$ centered at the origin $o=\{0\}\times \mathbb{S}^{m-1}/\sim$ with radius $r<R_h$, i.e.  the set  $[0,r)\times \mathbb{S}^{m-1}/\sim$, is rotationally invariant.  The volume $V(r)$ of the ball $B(o,r)$ and the volume $S(r)$ of the boundary $\partial B(o,r)$ are  given by \[ \begin{array}{lll}
V(r)=\omega_{m}\int_{0}^{r}h^{m-1}(s)ds &  {\rm and} &
S(r)= \omega_{m}h^{m-1}(r),
\end{array} \]respectively, where $\omega_m={\rm vol}(\mathbb{S}^{m-1})$.
The  Laplace operator on $B(o,r)$, expressed in polar coordinates, is given by \[\triangle= \frac{\partial^{2}}{\partial t^{2}}+ (n-1)\frac{h'}{h}\frac{\partial }{\partial t} + \frac{1}{h^{2}}\triangle_{\theta}.\]
To search for the Dirichlet eigenvalues $\lambda$ of $B(o,r)$ it is enough to seek smooth functions of the form  $u(t,\theta)=T(t)H(\theta)$ satisfying    $\triangle u + \lambda u=0$ in $B(o,r)$ and $u\vert \partial B(o,r)=0$, see \cite[p. 42]{chavel}. This  is equivalent to  the following eigenvalue problems
\begin{equation}\left\{\begin{array}{rll}
\triangle_{\theta} H+ \nu H =0 &{\rm in}&\mathbb{S}^{m-1}(1)\\
T''+(m-1) \displaystyle\frac{h'}{h}T' + (\lambda - \displaystyle\frac{\nu}{h^2})T=0&{\rm in}& [0,r] \end{array}\right.\end{equation} with initial conditions
 $T'(0)=0$ if $l=0$, $ T(t)\sim c\cdot t^l$  as $ t\to 0$  when $l=1,2\ldots$ and $T(r)=0$.
Here $T'=\partial T/\partial t$, $T''=\partial^{2} T/\partial t^{2}$.
\vspace{2mm}

 For each value $\nu_l=l(l+m-2)$, the set of all $\lambda$ such that the equation \begin{equation}T''+(n-1) \displaystyle\frac{h'}{h}T' + (\lambda - \displaystyle\frac{\nu_l}{h^2})T=0\label{eqseq0}\end{equation} has a non-trivial solution satisfying the initial conditions consist of an increasing  sequence of positive real numbers $\{ \lambda_{l,j}\}_{j=1}^{\infty}$,   \begin{equation} 0\leq \lambda_{l,1}< \lambda_{l,2}<\cdots \uparrow +\infty.\label{eqseq}\end{equation}  Moreover, each $\lambda_{l,i}$  determine a $1$-dimensional space of solutions, say, generated by $T_{l,i}$. The sequence \eqref{eqseq} is called the $\nu_{l}$-spectrum, without repetitions, of $B(o,r)$, denoted by $\sigma^{l}(B(o,r))$. 
\vspace{2mm}

 It is well known that the set of   eigenvalues of the sphere $\mathbb{S}^{m-1}$ are given by  $\nu_{l}=l(l+m-2)$, $l=0,1,2,\ldots$ and their multiplicity of each $\nu_{l}$ is   given by \[\delta(l,m)=\left( \begin{array}{c}m-1+l\\ l \end{array}\right)-\left( \begin{array}{c}m-2+l\\ l-1 \end{array}\right).\] Thus, there exists an orthonormal basis formed by eigenfunctions $H_{l,1}(\theta), \ldots, H_{l,\delta}(\theta)$ of the vector space $V_{\nu_l}=\{ \phi\colon \triangle_{\theta}\phi+\nu_{l}\phi=0\}$. This implies that the set of functions \[\{T_{l,i}(t)H_{l,1}(\theta),T_{l,i}(t)H_{l,2}(\theta), \ldots T_{l,i}(t)H_{l,\delta}(\theta)\}\] is an orthonormal basis of the vector space $\{\psi\colon \triangle \psi + \lambda_{l,i}\psi=0\}$.
Therefore, the multiplicity of each eigenvalue $\lambda_{l,i}$  of the sequence \ref{eqseq}, in the spectrum $\sigma (B(o,r)$ is $\delta(l,m)$. 
 Since all  of the  eigenvalues of $B(o,r)$ are obtained in this procedure above,  the spectrum of  $B(o,r)$, without repetitions, is the union of   the $\nu_{l}$-spectrum $\sigma^{l}(B(o,r))$, $l=0, 1, \ldots$ \[\sigma (B(o,r))=\cup_{l=0}^{\infty} \sigma^{l}(B(o,r))=\{ \lambda_{l,j}\}_{l=0,j=1}^{\infty,\,\, \infty},\]each $\lambda_{l,i}$ with multiplicity $\delta(l,m)$.

%
The details of this discussion can be found in  \cite[pp. 40-42]{chavel}. Observe that the eigenvalues associated to $\nu_0=0$ are those whose eigenfunctions are radial functions $u(t, \theta)=c\cdot T(t)$. We call them radial eigenvalues. More generally we have the following definition.
\begin{definition}Let $B(p,r)\subset M$ be  a geodesic ball  of radius $r>0$ and  center  $p$, (not necessarily rotational invariant). The  radial spectrum $\sigma^{{\rm rad}}(B(p,r))$
 of $B(p,r)$ is formed by those eigenvalues $\lambda_{k}$ of $\mathcal{L}=-\triangle\vert_{W_{0}^{2}(B(p,r))}$ whose associated eigenspace $V_k$   contains a  radial eigenfunction.
\end{definition}
The radial spectrum of a general  geodesic ball may be empty, however, if $B(p,r)$ is rotationally invariant, then its radial spectrum   is $\sigma^{\rm rad}(B(p,r))=\sigma^{l=0}(B(p,r))$. It should be remarked that there are non-rotationally invariant geodesic balls with non-empty radial spectrum, see Example \ref{examp1}.  A criteria due to S.Y. Cheng, P. Li and S.T. Yau    states that, for rotationally invariant  geodesic balls, in each eigenspace $V_{k}=\{\phi\colon \triangle \phi +\overline{\lambda}_{k} \phi=0\}$ either $\phi (p)=0$ for all $\phi \in V_{k}$ or $V_k$ has a radial eigenfunction and $\overline{\lambda}_{k}\in \sigma^{{\rm rad}}(\Omega) $, see    \cite[Lem. 7]{CLY}. 
As observed, the spectrum of  rotationally invariant geodesic balls of model $m$-manifolds can be decomposed into the union of subsets called $\nu_l$-spectrum associated to eigenvalues $\nu_l= l(l+m-2)$ of the sphere $\mathbb{S}^{m-1}$.  Moreover, the multiplicity of each eigenvalue of $\nu_l$-spectrum is  the multiplicity of the eigenvalue $\nu_l$ of the sphere. In particular, the radial eigenvalues, associated to the $\nu_{0}=0$, has multiplicity  one. If we take in consideration the Cheng-Li-Yau criteria   then we have the following theorem.
  \begin{theorem}Let $B(o,r)\subset \mathbb{M}_{h}^{m}$  be a geodesic ball of radius $r>0$ and  centred at the origin. Let  $V_k=\{\phi \colon \triangle \phi + \overline{\lambda}_{k}\phi =0\}$ be the eigenspace associated to the eigenvalue $\overline{\lambda}_{k}$. If ${\rm dim}V_{k}\geq 2$ then zero belongs to the nodal set of each $\phi \in V_{k}$. \label{thmNodal}
  \end{theorem}

\subsection{Stochastically incomplete  model manifolds}
 \!\! Let $M$ be a  Riemannian manifold and  $p_t(x,y)\in 
C^{\infty}\left((0, \infty)\times M \times M\right)$  be the heat kernel of 
$M$. It is well known that \[\int_{M}p_t(x,y)d\nu (y)\leq 1.\] This property 
allows to see the heat kernel as a probability distribution in the space of 
random paths on $M$. More precisely, for $x\in  M$ and $U\subset M$, open 
subset, $\int_{U}p_t(x,y)d\nu (y)$ is the probability that a random path 
emanating from $x$ lies
in $U$ at time $t$. Thus if we have strict inequality, 
$\int_{M}p_t(x,y)d\nu(y)<1$, then there is a positive
probability that a random path will reach infinity in finite time $t$. This 
motivates the following definition.

\begin{definition}A Riemannian manifold $M$ is stochastically complete if for 
every $x\in M$ and $t>0$  one has that 
\begin{equation}\label{heatkernel}\int_{M}p_{t}(x,y)d\nu(y)=1,
\end{equation} where $p_t(x,y)\in C^{\infty}((0, \infty)\times M \times M)$ is 
the heat kernel of $M$. Otherwise it is called stochastically incomplete.
\end{definition}  The
 spectrum of a stochastically incomplete model manifolds $\mathbb{M}_{h}^{m}$ is  discrete, see \cite[Example 6.12]{bessa-pigola-setti}, say \[\sigma(\mathbb{M}^{m}_{h})=\{0<\lambda_{1}(\mathbb{M}_{h}^{m})< \lambda_{2}(\mathbb{M}_{h}^{m})\leq \cdots\}\] and let \[\sigma(B(o,r))=\{0<\lambda_{1}(B(o,r))< \lambda_{2}(B(o,r))\leq \cdots\}\] be the spectrum  (with repetition) of  $B(o,r)\subset \mathbb{M}_{h}^{m}$. It is well known that
 the ${\rm k}^{th}$-eigenvalue $\lambda_k(\mathbb{M}^{m}_{h})$  is  obtained as the limit
\begin{equation}\lambda_k(\mathbb{M}^{m}_{h})=\lim_{r\to \infty}\lambda_k(B(o,r))\label{eqmari}
\end{equation} for $k=1,2,\ldots$.  Moreover,  \eqref{eqmari} holds  for any geodesically complete Riemannian with discrete spectrum and  more general operators. In particular,  the operator  $L_0(T)=T''+(m-1)\displaystyle\frac{h'}{h}T'$ on $[0,r]$, see   \cite[eq. 2.85]{bmr}. Since 
 the radial spectrum  $\sigma^{\rm rad} (B(o,r))=\{ 0< \lambda_{1}^{\rm rad} (B(o,r))\leq  \lambda_{2}^{\rm rad} (B(o,r))\leq  \cdots\}$  is the spectrum of the operator $L_0(T)=T''+(m-1)\displaystyle\frac{h'}{h}T'$ on $[0,r]$ then  the ${\rm i}^{th}$ radial eigenvalue of $\mathbb{M}^{m}_{h}$ is also obtained as a limit $\lambda_i^{\rm rad}(\mathbb{M}^{m}_{h})=\lim_{r\to \infty}\lambda_i^{\rm rad}(B(o,r))$ for $i=1,2,\ldots$.

There exists  
a simple and elegant geometric criteria for stochastic incompleteness of model manifolds $\mathbb{M}_{h}^{m}$,  proved by many authors in different settings. 
\begin{theorem}[\cite{ahlfors},  \cite{grigoryan89}, \cite{Gr}, \cite{ichihara1}] A geodesically complete model manifold $\mathbb{M}_{h}^{m}$ is stochastically incomplete if and only if \[ \int_{0}^{\infty} \frac{V(t)}{S(t)}dt < \infty.\]Where $V(r)=\omega_{m}\int_{0}^{r}h^{m-1}(s)ds$   and $
S(r)= \omega_{m}h^{m-1}(r)$.\label{thm1.4}
\end{theorem}

\subsection{\textquotedblleft Harmonic series\textquotedblright of radial eigenvalues }
Our  main result in this section regards  the radial spectrum of rotationally invariant balls of model manifolds. The  radial spectrum of a stochastically incomplete model manifolds $\mathbb{M}_{h}^{m}$ is deeply related with this geometric criteria as shows our first result, Theorem \ref{thmMain1-intro}. 
 \begin{theorem}\label{thmMain1-intro} Let 
 $B(o,r)$ be of the model $\mathbb{M}_{h}^{m}$, centred at the origin $o$  of radius $r$, with $\lambda_{1}(B(o,r))>0$. Let $\sigma^{\rm 
rad}(B(o,r))=\{\lambda_{1}^{\rm rad}(B(o,r))< \lambda_{2}^{\rm rad}(B(o,r)) 
<\cdots\}$ be the radial spectrum of $B(o,r)$. Then 
\begin{equation}\label{eqMain1-intro}\sum_{i=1}^{\infty}\frac{1}{\lambda_{i}^{\rm 
rad}(B(o,r))}=\int_{0}^{r}\frac{V(s)}{S(s)}ds\cdot
 \end{equation}If the model $\mathbb{M}_{h}^{m}$ is stochastically 
incomplete  then 
 \begin{equation}\label{eqMain2-intro}\sum_{i=1}^{\infty}\frac{1}{\lambda_{i}^{\rm 
rad}(\mathbb{M}_{h}^{n})}=\int_{0}^{\infty}\frac{V(s)}{S(s)}ds\cdot
 \end{equation}
\end{theorem}
\begin{corollary}Let $B(o,1)\subset \mathbb{M}^{2}_{h}$ be a geodesic ball of radius $r=1$ in the model manifold $\mathbb{M}^{2}_{h}$, where   $h(t)= \sinh(t), t, \sin(t)$, respectively. Then,
\begin{eqnarray}\sum_{i=1}^{\infty}\frac{1}{\lambda_{i}^{\rm 
rad}(B(o,1))}&=& \log\left(\frac{\left(1+e\right)^2}{4e}\right)\approx 0.240229\,\,\,\,\text{if}\,\,h(t)=\sinh (t)\nonumber \\
\sum_{i=1}^{\infty}\frac{1}{\lambda_{i}^{\rm 
rad}(B(o,1))}&=&0.25\,\,\,\,\text{if}\,\,h(t)=t \nonumber \\
\sum_{i=1}^{\infty}\frac{1}{\lambda_{i}^{\rm 
rad}(B(o,1))}&=&\log\left(\sec(\frac{1}{2})\right)\approx 0.261168\,\,\,\,\text{if}\,\,h(t)=\sin(t). \nonumber
\end{eqnarray}
\end{corollary}
 When the model manifold $\mathbb{M}_{h}^{m}$ is the Euclidean space $\mathbb{R}^{m}$, we can show that the \textquotedblleft harmonic series\textquotedblright of  the eigenvalues, (without repetitions),  $\{\lambda_{l,i}\}_{i=1}^{\infty}=\sigma^{l}(B(o,r))$, for $l=0,1,\ldots$, also converges.
\begin{theorem}\label{thmMain2}Let 
 $B(o,r)$ be the 
geodesic ball of $\mathbb{R}^{m}$ with  radius $r$ centred at the origin $o$ . Let $\sigma^{l}(B(o,r))=\{\lambda_{l,1}(B(o,r))< \lambda_{l,2}(B(o,r)) 
<  \cdots\}$ be the $\nu_l$-spectrum of $B(o,r)$ without repetition, $l=0,1,\ldots.$ Then
\begin{equation}\label{eqMain3}\sum_{i=1}^{\infty}\frac{1}{\lambda_{l,i}(B(o,r))}=\left(\frac{1}{1+2\frac{l}{m}}\right)\int_{0}^{r}\frac{V(s)}{S(s)}ds=\left(\frac{1}{1+2\frac{l}{m}}\right)\frac{r^2}{2m}
 \end{equation} and
 \begin{equation}\label{eqMain4}\sum_{i=1}^{\infty}\frac{1}{\lambda_{l,i}^{2}(B(o,r))}=\frac{r^4}{2(2l+m)^2(2+2l+m)}\cdot
 \end{equation}

\end{theorem}

If we let $\sigma(B(o,r)) =\{0< \lambda_1<\lambda_2 \leq \cdots\}$ be the spectrum of be  $B(o,r)\subset \mathbb{R}^m$, repeating the eigenvalues according to their multiplicities, in this case,  $\sigma(B(o,r))$ is a union of $\delta (l,m)$-copies of $\sigma^{l}(B(o,r))$ for $l=0,1,\ldots$, then  we can compute the whole sum  \[
\sum_{k=1}^\infty\displaystyle\frac{1}{\lambda_k^2}=\sum_{l=0}^\infty\left(\sum_{i=1}^\infty\frac{\delta(l,m)}{\lambda_{l,i}^2}\right)=
\sum_{l=0}^\infty \delta(l,m)\cdot\displaystyle\frac{r^4}{2(2l+m)^2(2+2l+m)}\cdot
\]In the particular case of dimension $m=2$ or dimension $m=3$,
\begin{equation}\label{eq:2.5}\begin{array}{llllll}& \displaystyle\sum_{l=0}^\infty \delta(l,2)\cdot\displaystyle\frac{r^4}{2(2l+2)^2(2+2l+2)}&=&\displaystyle\frac{\pi^2-6}{96}\cdot r^4 \\
& \displaystyle\sum_{l=0}^\infty \delta(l,3)\cdot\displaystyle\frac{r^4}{2(2l+3)^2(5+2l)}&=&\displaystyle\frac{12-\pi^2}{64}\cdot r^4\end{array}                     \end{equation}
 However, this series diverges for $m\geq 4$ in agreement with  Theorem \ref{ThmMainA} equation \eqref{eq5.3}. The convergence of the series is related to the fact that the Green function belongs to $L^2$ only in the case of $m=1,2,3$. We should remark that the divergence in higer dimension is because of the multiplicity of the eigenvalues. In fact, if we consider the spectrum  $\widetilde\sigma(B(o,r)) = \cup_{l=1}^{\infty} \sigma^{l}(B(o,r))=\{ 0< \overline{\lambda}_1 < \overline{\lambda}_2 <  \cdots \} $ without repetition then \begin{equation}\label{suma04}
\sum_{k=1}^\infty\displaystyle\frac{1}{\widetilde\lambda_{k}^2}=\sum_{l=0}^{\infty}\left(\sum_{i=1}^\infty\displaystyle\frac{1}{\lambda_{l,i}^2}\right)=\sum_{l=0}^{\infty}\frac{r^4}{2(2l+m)^2(2+2l+m)}<\infty.
\end{equation}Morever, $\displaystyle\lim_{m\to \infty}\sum_{k=1}^\infty\displaystyle\frac{1}{\widetilde\lambda_{k}^2}=0$.
\begin{corollary}\label{cor2.2}Let 
 $B(o,r)$ be the 
geodesic ball of $\mathbb{R}^{m}$ with  radius $r$ centred at the origin $o$ . Let $\sigma^{l}(B(o,r))=\{\lambda_{l,1}(B(o,r))\leq \lambda_{l,2}(B(o,r)) 
\leq \cdots\}$, $l=0,1,\ldots$. Then  
\begin{equation}
\lambda_{l,k}(B(o,r))\geq \frac{k(m+2l)}{m}\cdot \frac{1}{\int_{0}^{r}\frac{V(s)}{S(s)}ds}= \frac{2k(m+2l)}{r^2}\cdot
\end{equation}
\end{corollary}
Corollary \ref{cor2.2} should be compared with the estimates for the first eigenvalue obtained in \cite[Thm.2.1]{barroso-bessa} and \cite[Thm.2.7]{bessa-montenegro-2009}. We believe that Theorem \ref{ThmMainA} is valied for rotationally invariant geodesic balls of model manifolds. We formalize this in the following conjecture
\begin{conjecture}There exists a  constant $c=c(m, l)<1$ such that for a geodesic ball $B(o,r)\subset \mathbb{M}_{h}^{m}$ the harmonic series  converges \[\sum_{i=1}^{\infty}\frac{1}{\lambda_{l,i}(B(o,r))}=c\cdot\int_{0}^{r}\frac{V(s)}{S(s)}ds, \,\,l=0, 1,2,\ldots \]
\end{conjecture}
\subsection{Example}\label{examp1}

 Let $\{\partial/\partial x,\pa y,\pa z\}$ be a globally defined non-zero
vector fields on $\mathbb{S}^{3}$ satisfying these conditions  \[[\pa x, \pa y]=2\,\pa x,\,\, [\pa y, \pa z]=2\,\pa x,\,\,
[\pa z,\pa x]=2\,\pa y\] and let $dx$, $dy$ and $dz$ be their dual 1-forms. Consider, on $\Omega=[0, r] \times \mathbb{S}^{3}/\sim$, where $(t, \theta)\sim (s,\beta) \Leftrightarrow t=s=0$ or $t=s$ and $\theta=\beta$, the following metric
\[ds^{2}=dt^{2}+a^{2}(t,\theta)dx^{2}+b^{2}(t,\theta)dy^{2}+c^{2}(t,\theta)dz^{2},\] where   $a, b,
c\colon [0, r] \times \mathbb{S}^{3}\to \mathbb{R}$ are defined by
\[ \begin{array}{lll}a(t,\theta)&=&\sinh^{2}[t]/t\cdot f(t,\theta)\\ b(t,\theta)&=&t\cdot f(t,\theta)\\ c(t,\theta)&=&\sinh[t]\end{array}\]with $f\colon [0,r]\times \mathbb{S}^{3} \to (0,\infty)$   given by $f(t,\theta)=1+t^{k}q(\theta)$, $k\geq 3$ and $q\colon \mathbb{S}^{3}\to (0,\infty)$ is smooth. This metric is clearly smooth in $(0,r]\times \mathbb{S}^{3}/\sim$. We need only check that it is smooth at the origin $\{0\}\times \mathbb{S}^{3}\sim$. The coefficients  near $t=0$ and every $\theta \in \mathbb{S}^{3}$ fixed are given by   $a(t,\theta)=t+ t^3/3+ 2t^5/45+ O(t^6)$, $b(t,\theta)=t+ t^3/6+ t^5/120+ O(t^6)$, $c(t, \theta)=t+O(t^6)$. This shows that $ds^2\approx \text{can}_{\mathbb{H}^{4}}$ in the $C^{2}$-topology as $t\approx 0$, where   $\text{can}_{\mathbb{H}^{4}}=dt^{2}+\sinh^2[t]\left(dx^{2}+dy^{2}+dz^{2}\right)$ is the canonical metric on the hyperbolic space $\mathbb{H}^{4}(-1)$. Let $\Omega=B(o,r)$ be the  geodesic ball with center at the origin $o=\{0\}\times \mathbb{S}^{3}/\sim $ and radius $r$ with respect to the metric $ds^{2}$. It is clear that $\Omega$  is not rotationally symmetric. The Laplace operator $\triangle_{ds^2}$ is given by 
\begin{eqnarray}
\triangle_{ds^2}(t,\theta)&=& \frac{\partial^2}{\partial t^2}+ 3 \coth(t) \frac{\partial}{\partial t}\nonumber \\
&& \nonumber \\
&& + \left[\frac{t^2(1+t^{k}q(\theta))^2}{\sinh^4(t)}\right] \frac{\partial^{2}}{\partial x^{2}} +\, \left[\frac{1}{t^2(1+t^{k}q(\theta))^2}\right] \frac{\partial^{2}}{\partial y^{2}} + \frac{1}{\sinh^2(t)}\frac{\partial^{2}}{\partial z^{2}}\nonumber \\
&& \nonumber \\
 && +\, \left[\frac{2t^{2+k} (1+t^{k}q(\theta))}{\sinh^{2}(t)}\right]\frac{\partial q}{\partial x}\frac{\partial}{\partial x}- \left[\frac{2t^{k-2}}{(1+t^{k}q(\theta))^3}\right]\frac{\partial q}{\partial y}\frac{\partial}{\partial y}\cdot\nonumber 
\end{eqnarray}Let $\widetilde{\Omega}=([0,r)\times \mathbb{S}^{3}/\sim, \text{can}_{\mathbb{H}^{4}})$ be the geodesic ball of radius $r$ centered at the origin in the Hyperbolic space $\mathbb{H}^{4}(-1)$. The Laplace operator $\triangle_{\text{can}_{\mathbb{H}^{4}}}$ on $\widetilde{\Omega}$ is given by \begin{eqnarray}
\triangle_{\text{can}_{\mathbb{H}^{4}}}(t,\theta)&=& \frac{\partial^2}{\partial t^2}+ 3 \coth(t) \frac{\partial}{\partial t}+ \frac{1}{\sinh^{2}(t)}\left(\frac{\partial^{2}}{\partial x^{2}}+ \frac{\partial^{2}}{\partial y^{2}}+  \frac{\partial^{2}}{\partial z^{2}}\right).\nonumber 
\end{eqnarray}
Observe that  $\triangle_{ds^2} =\triangle_{\text{can}_{\mathbb{H}^{4}}}$ on the set of the smooth radial functions $u(t, \theta)=u(t)$ with $u'(0)=0$, in particular,
 $\Omega$ and $\widetilde{\Omega}$ has the same radial eigenfunctions and radial  eigenvalues. 
 
 \vspace{2mm}
 
  Thus, $(\Omega, ds^2)$ is an example of a non-rotationally symmetric  geodesic ball  with the same radial spectrum of a rotationally symmetric ball $(\tilde{\Omega}, \text{can}_{\mathbb{H}^{4}})$. We should notice that the volume functions  $t\to V(t)$ and $t\to S(t)$ are the same for both metrics,  that can be checked directly, coherently with the identity \eqref{eqMain1-intro}. This example  is a  variation of the example \cite[Examp. 2]{bessa-montenegro-2008} which,  by its turn, was inspired by the example of G. Perelman     in \cite{perelman}.

\section{Spectrum of extrinsic balls of minimal submanifolds}\label{sec3}Let $\varphi \colon M \to \mathbb{R}^{n}$ be  a  proper and minimal
immersion of a  complete  $m$-dimensional Riemannian manifold into $\mathbb{R}^{n}$. Let $\Omega_r= 
\varphi^{-1}(B(0,r))$ be the extrinsic ball of radius $r$. It was proved 
in \cite{bessa-montenegro-2007} and \cite{CLY} that the first Dirichlet eigenvalue of 
$\Omega_r$ is bounded below as \begin{equation}\label{eq2.1} 
\lambda_1(\Omega_r) \geq 
\lambda_1(B(0,r))=\frac{c^{2}_m}{r^{2}},\end{equation} where 
$c_m$ is a constant depending only on the dimension of $M$ and 
$\lambda_1(B(0,r))$ is the first Dirichlet eigenvalue of the ball 
of radius $r$ in the Euclidean $m$-space $\mathbb{R}^{m}$.  This inequality can be read as  
\begin{equation}\label{eq2.2} \frac{1}{\lambda_{1}^{2}(\Omega_r)}\leq C_m \cdot 
r^4,\end{equation} where $C_m=1/c^{4}_m$. By Theorem \ref{ThmMainA} we have that 
$\sum_{k=1}^{\infty}\frac{1}{\lambda_{k}^{2}(\Omega_r)}< \infty$ in dimensions 
$m=2,3$. For a totally geodesic $\mathbb{R}^m\subset \mathbb{R}^n$ with $m=2,3$ we have, see  (\ref{eq:2.5}),  
$$
\sum_{k=1}^{\infty}\frac{1}{\lambda_{k}^{2}(\Omega_r)}=C_m r^4
$$
where $C_m$ is a constant depending only on the dimension. For a non-totally geodesic minimal submanifold we will give lower and upper bounds for this series in 
terms of the volume  ${\rm vol}(\Omega_r)$, of the radius $r$ and the dimension $m=2,3$. We have the following result.
\begin{theorem}\label{thmMark}Let $\varphi \colon M \to \mathbb{R}^{n}$ be  a 
complete   Riemannian $m$-manifold, properly and minimally  immersed  
into the Euclidean $n$-space. Given $o\in M$ and let $\Omega_r$ be the 
extrinsic $r$-ball with center $o$, namely,
\[\Omega_r=\varphi^{-1}\left(B(\varphi(o),r)\right),\]
If $m=2,3$, then
\begin{equation}
A_m \cdot\omega_m\cdot\left(\frac{r^m}{{\rm vol}(\Omega_r)}\right)\cdot r^4\leq 
\sum_{k=1}^\infty\frac{1}{\lambda_k^2(\Omega_r)}\leq B_m\cdot\zeta(4/m)\cdot \left(\frac{{\rm 
vol}(\Omega_r)}{r^m}\right)^{4/m}\cdot r^4,\label{eq2.3Gimeno}
\end{equation}where $A_m=\frac{1}{4m^2}\left(1+ 
\frac{m}{4+m}-\frac{2m}{2+m}\right)$ and 
$B_m=\frac{e^{4/m}}{16\pi^{2}}$ are constants,  $\omega_{m}$ is the volume of the geodesic ball of radius $1$ in $\mathbb{R}^m$ and $\zeta 
(4/m)=\sum_{k=1}^{\infty}\frac{1}{k^{4/m}}\cdot$
\end{theorem}
If  $\varphi \colon M\to \mathbb{R}^{n}$ is a minimal $m$-submanifold with  finite $L^m$-norm of  the second fundamental form $\alpha$,
\[
\int_M\Vert \alpha\Vert^md\mu<\infty,
\]
 then $M$ has finite number of ends $\{{\rm End}_1,\cdots,{\rm End}_k\}$, see \cite{anderson, JM}. Moreover, the function $r\to \displaystyle \frac{{\rm vol}(\Omega_r)}{\omega_{m}r^{m}}$ is increasing, \cite{MP-2012, Palmer} and \[\lim_{r\to \infty} \frac{{\rm vol}(\Omega_r)}{\omega_{m}r^{m}}=\mathcal{E}\]where $\mathcal{E}$ is related with the finite number of ends of $M$ in the following way:
$$
\mathcal{E}=\left\{
\begin{array}{ccl}
  \displaystyle\sum_{i=1}^kI_i&\text{ if }&m=2\\
  k&\text{ if }&m=3,
\end{array}\right.  
$$
where  $I_i$ is the geometric index of the end ${\rm End}_i$, see \cite{anderson,JM,  Che4}. 
\begin{corollary} Let $\varphi \colon M\to \mathbb{R}^{n}$ be  a 
complete   Riemannian $m$-manifold, properly and minimally  immersed  
into the Euclidean $n$-space with  finite $L^m$-norm of   the second fundamental form $\alpha$,
\[
\int_M\Vert \alpha\Vert^md\mu<\infty.
\]If $m=2,3$
 then
\begin{equation}\label{eq:2.11}
\left(\frac{A_m}{\mathcal{E}}\right)\cdot r^4\leq 
\sum_{k=1}^\infty\frac{1}{\lambda_k^2(\Omega_r)}\leq B_m\cdot\zeta(4/m)\cdot \left(\omega_m\mathcal{E}\right)^{4/m}\cdot r^4
\end{equation}
\end{corollary}

\section{Weighted Green operator and the Dirichlet spectrum}\label{sec4} A weighted manifold is a triple  
$(M, ds^{2}, \mu)$ consisting of  a Riemannian  manifold $(M,ds^{2})$  and a 
measure $\mu$ with  positive density function $\psi\in C^{\infty}(M)$, i.e.,  $d\mu=\psi d\nu,$ where $d\nu$ is the Riemannian density of 
$(M,ds^{2})$.
Let $\Omega\!\subset \!M$ be a relatively compact open subset with smooth boundary $\partial 
\Omega \neq \emptyset$ and consider 
the weighted Laplace operator $\triangle_{\mu}\colon C_{0}^{\infty}(\Omega)\to 
C_{0}^{\infty}(\Omega)$, acting on   $C_{0}^{\infty}(\Omega)$, the space of 
smooth 
functions with compact support on $\Omega$,  defined  by  
\[\triangle_{\mu}=\frac{1}{\psi}\diver (\psi\cdot \grad).\] The weighted Laplace operator is densely defined 
and  symmetric with respect to  $L^{2}(\Omega, \mu)$-inner product, but it   
is not self-adjoint. As in the classic case, we may consider the Sobolev spaces $W^{1}_{0}(\Omega,\mu)$ 
 as the  closure of 
$C^{\infty}_{0}(\Omega)$ with respect to the norm 
\[\Vert u \Vert^{2}_{W^{1}(\Omega, \mu)}\colon = \int_{\Omega}u^{2}d\mu + 
\int_{\Omega}\vert \grad u \vert^{2}d\mu\] and $W_{0}^{2}(\Omega, \mu)$, formed by 
those functions $u\in W^{1}_{0}(\Omega,\mu)$ whose weak Laplacian 
$\triangle_{\mu}u $ exists and belongs to $L^{2}(\Omega,\mu)$, i.e. 
\[W_{0}^{2}(\Omega, \mu)=\{u\in W^{1}_{0}(\Omega,\mu)\colon \triangle_{\mu}u 
\in L^{2}(\Omega, \mu)\}.\] Turns out that the operator 
$\mathcal{L}=-\triangle_{\mu}\vert_{W_{0}^{2}(\Omega, \mu)}$ is a self-adjoint, non-negative definite elliptic operator and its 
 spectrum, denoted by $\sigma(\Omega, \mu)$, is a discrete increasing 
sequence of non-negative real numbers
$\{\lambda_{k}\}_{k=1}^{\infty}\subset [0, \infty)$,
(counted according to multiplicity),  with $\lim_{k\to \infty} 
\lambda_{k}=\infty$, see \cite[Thm.10.3]{grigoryan-book}. The weighted Laplace operator $\triangle_\mu $ extends the classical Laplace operator $\triangle$, in the sense that if the density function $\psi\equiv 1$ then $\triangle_\mu=\triangle$. Let $(\Omega, \mu)$ be a weighted bounded open  subset with smooth boundary $\partial 
\Omega \neq \emptyset$ of a Riemannian weighted manifold $(M,ds^2, \mu)$.  The (weighted) Green operator $G^{\Omega}\colon L^2(\Omega, \mu)\to 
L^{2}(\Omega, \mu)$ is given by 
\begin{eqnarray}G^{\Omega}(f)(x)&=&\int_{0}^{\infty}\int_{\Omega}p_t^{\Omega}(x,y)f(y)d\mu(y)dt\nonumber \\
&=&\int_{\Omega} g(x,y)f(y)dy,\end{eqnarray}where $p_t^{\Omega}(x,y)$ is the heat kernel of the operator $\mathcal{L}=-\triangle_{\mu}\vert_{W_{0}^{2}(\Omega, \mu)}$ and
\begin{equation}
g^{\Omega}(x,y)=\int_{0}^{\infty}p_{t}^{\Omega}(x,y)dt
\end{equation} is the Green function of $\Omega$. The Green operator  is a bounded  self-adjoint operator in $L^{2}(\Omega, \mu)$ and it is the inverse of $\mathcal{L}$, i.e., $G^{\Omega}=\mathcal{L}^{-1}$. Thus for any $f\in L^{2}(\Omega, 
\mu)$ there is a unique solution $u=G^{\Omega}(f)\in W^{2}_{0}(\Omega, \mu)$ to 
the equation $-\triangle_{\mu}u=f$. If $f\in C^{\infty}_{0}(\Omega)$ then 
$G^{\Omega}(f)\in C^{\infty}(\Omega)$ see details in   \cite[Thm. 13.4]{grigoryan-book}. Applying $G^{\Omega}$ to the equation 
\[\triangle_{\mu}u+\lambda_i(\Omega)u=0 \] we obtain that the $i$-th eigenvalue 
$\lambda_{i}(\Omega)$ of $\Omega$  is given  by $ 
\lambda_{i}(\Omega)=\displaystyle u/G^{\Omega}(u)\cdot$  The difficulty  to 
know precisely the $i^{\rm th}$-eigenvalue applying the 
Green operator is that one needs to know an  $i^{\rm th}$ eigenfunction to start. 
For simplicity of notation, if no ambiguity arises,  we will suppress the 
reference to $\Omega$ in the gadgets associated to $\Omega$.
Thus, $g$ and $G$ will be, respectively, the Green function and the Green  
operator of the operator $\mathcal{L}=-\triangle_{\mu}\vert_{W_{0}^{2}(\Omega, 
\mu)}$.  Our main result in this section is that in order to obtain the first 
Dirichlet eigenvalue, assuming $\lambda_{1}(\Omega)>0$, one  picks a positive function $f\in L^{2}(\Omega, \mu)$ 
and compute the limit
\[\lambda_{1}=\lim_{k\to \infty}\displaystyle\frac{ \Vert G^{k}(f)\Vert_{2}
}{\Vert G^{k+1}(f)\Vert_{2}},\] where $\Vert u \Vert_{2}=\int_{\Omega}\vert u 
\vert^{2}d\mu$ is the $L^{2}(\Omega,\mu)$-norm and
 $G^{k}=\stackrel{k-times}{\overbrace{G\circ  \cdots \circ G}}$. More 
generally, let $f\in L^{2}(\Omega, \mu)$ and let $\ell$ be the smallest positive 
integer such that \[\int_{\Omega} f(x) u_{i}(x)\,d\mu(x) =0\] for $i=1,2,\ldots, \ell-1$ and  
$\int_{\Omega} f(x) u_{\ell}(x)d\mu(x) \neq 0$, where 
$\{u_\ell\}$ is  an orthonormal basis of  $L^{2}(\Omega, \mu)$ formed by  
eigenfunctions $u_\ell$ with eigenvalues $\lambda_{\ell}$.
  Then   the $\ell^{th}$-eigenvalue is given by \[\lambda_{\ell}=\lim_{k\to 
\infty}\displaystyle\frac{ \Vert G^{k}(f)\Vert_{2}
}{\Vert G^{k+1}(f)\Vert_{2}}\cdot\]
\begin{theorem}\label{Main2}Let $(\Omega, ds^2, \mu)$ be a weighted 
 bounded open subset  of a weighted  $m$-manifold $(M, ds^2, \mu)$, with smooth 
boundary $\partial \Omega\neq \emptyset$.  Let $G$ be the Green operator of $\mathcal{L}$ and let $\{u_i\}$ be an orthonormal basis of  $L^{2}(\Omega, \mu)$ 
formed with   eigenfunctions $u_i$ with eigenvalues $\lambda_{i}$.
Then, for any $f\in L^2(\Omega)$,
\begin{equation}\label{eq3.3}
\lim_{k\to\infty}\frac{\Vert G^k(f) \Vert}{\Vert G^{k+1}(f) \Vert}=\lambda_\ell,
\end{equation}where $\ell$ is the first positive integer such that,
\begin{equation}\int_{\Omega} f(x) u_{i}(x)\,d\mu(x) =0\,\,\,\text{for}\,\,\, i=1,2,\ldots, \ell-1\,\,\,\text{ and  }
\int_{\Omega} f(x) u_{\ell}(x)d\mu(x) \neq 0. \label{eq1.9}
\end{equation}
Moreover,
\begin{equation}\label{eq3.5}
\lim_{k\to\infty}\frac{G^k(f)}{\Vert G^{k}(f) \Vert}=\phi_\ell\in 
\ker(\triangle_\mu+\lambda_\ell)\,\,\,\text{in } L^2(\Omega,\mu). 
\end{equation}
In particular, for any positive $f\in L^2(\Omega,\mu)$,
\begin{equation}
\lim_{k\to\infty}\frac{\Vert G^k(f) \Vert}{\Vert G^{k+1}(f) \Vert}=\lambda_1 
\,\,\,{\rm and} \,\,\,\lim_{k\to\infty}\frac{G^k(f)}{\Vert G^{k}(f) \Vert}=   
u_1\,\, \text{in}\,\, L^2(\Omega,\mu). 
\end{equation} In addition, if 
 $f\in L^{2}(\Omega, \mu)$ satisfying \eqref{eq1.9} and  denoting by $f_1$ the function defined    by
\begin{equation}
f_1=f-\lambda_lG(f),
\end{equation}
then,
\begin{equation}
\begin{aligned}
&\lim_{k\to\infty}\frac{\Vert G^k(f_1)\Vert}{\Vert G^{k+1}(f_1)\Vert}=\lambda_n\\
\end{aligned}
\end{equation}
with
$$
\lambda_n\geq\lambda_l.
$$
And
\begin{equation}
\begin{aligned}
\lim_{k\to\infty}\frac{G^k(f_1)}{\Vert G^{k}(f_1)\Vert}=\phi_n\in \ker(\triangle_\mu+\lambda_n).
\end{aligned}
\end{equation}
\end{theorem}

 Theorem \ref{Main2} is effective if 
 $\Omega=B(o,r) 
 $ is a rotationally invariant geodesic ball of a model manifold  $\mathbb{M}_{h}^{m}$ and $\triangle_{\mu}=\triangle$. Indeed,  let
$C^{\rm rad}(B(o,r))$ be the set of continuous radial functions  $\phi\colon 
B(o,r)\to \mathbb{R}$, $\phi(t,\theta)=\phi(t)$. Define the operator $T\colon C^{\rm 
rad}(B(o,r))  \to C^{\rm rad}(B(o,r))$ by 
\begin{equation}T(\phi)(t,\theta)=\int_{t}^{r}\frac{\int_{0}^{\sigma} 
h^{n-1}(s)\phi(s)ds}{h^{n-1}(\sigma)}d\sigma.\label{T}
\end{equation}
Observe  that $T$ is the Green operator for $L_0(f)= f''+(m-1) \displaystyle\frac{h'}{h}f'=\triangle_{\mu}$ for the measure $d\mu =\omega_{m}h^{m-1}d\nu$, hence Theorem \ref{Main2} can be applied to obtain  the following  theorem.

\begin{theorem}\label{theo3}
Let $B(o,r)\subset \mathbb{M}_{h}^{m}$ be rotationally invariant geodesic ball of a model manifold with boundary $\partial B(o,r)\neq \emptyset$. For any 
$\phi\in C^{\rm rad}(B(o,r))$ we have
\begin{eqnarray}\frac{\Vert T^{k}\phi\Vert_{L^2}}{\Vert T^{k+1}\phi\Vert_{L^2}}\to \lambda_l^{\rm rad}(B(o,r)) &\text{and}&  \displaystyle\frac{T^{k}\phi}{\Vert T^{k}\phi\Vert_{L^2}}\to \phi_{l} \text{  in }L^2,
\end{eqnarray}as $ k\to \infty.$ With 
$
\phi_l\in\ker (\triangle+\lambda_l),
$
 $\lambda_l=\lambda_l(B(o,r))$ is $l^{\rm th}$
radial eigenvalue, where $l$ is the first integer such that 
\begin{equation}\int_{B(o,r)} \phi\, u_{i}\,d\nu =0\,\,\,\text{for}\,\,\, i=1,2,\ldots, \ell-1\,\,\,\text{ and  }
\int_{B(o,r)}\phi\, u_{\ell}d\nu\neq 0,
\end{equation}where $\{u_i\}$ is an orthornormal basis of $L^{2}([0,r], \mu)$ eigenfunctions of $L_0$.
Thus, in particular, for any \emph{positive} 
$\phi\in C^{\rm rad}(B(o,r))$ we have
\begin{eqnarray}\frac{\Vert T^{k}\phi\Vert_{L^2}}{\Vert T^{k+1}\phi\Vert_{L^2}}\to \lambda_1(B(o,r)) &\text{and}&  \displaystyle\frac{T^{k}\phi}{\Vert T^{k}\phi\Vert_{L^2}}\to \phi_{1} \text{  in }L^2,
\end{eqnarray}as $ k\to \infty.$
\end{theorem}
To show  how efficient Theorem \ref{theo3} is, we  let
 $B(o,r)$ be the geodesic ball centered at the origin $o$ and  radius $r$ in rotationally symmetric manifold $\mathbb{M}^{n}_{h}$. Consider the following two maps:
$$
\begin{aligned}
\mathcal{T}^k \colon C^{\rm rad}(B(o,r))\to \mathbb{R},&\quad 
\mathcal{T}^k(\phi):=\frac{\Vert T^{k}\phi\Vert_{L^2}}{\Vert T^{k+1}\phi\Vert_{L^2}}\\
\mathcal{T}^\infty: C^{\rm rad}(B(o,r))\to \mathbb{R},&\quad 
\mathcal{T}^\infty(\phi):=\lim_{k\to\infty}\frac{\Vert T^{k}\phi\Vert_{L^2}}{\Vert 
T^{k+1}\phi\Vert_{L^2}},
\end{aligned}
$$
and the following family of functions arising from the $\phi_0=1$ 
\begin{equation}
\left\{
\begin{aligned}
\phi_0&=1\\
\phi_k&=\phi_{k-1}-\left(\mathcal{T}^\infty(\phi_{k-1})\right)T\phi_{k-1}.
\end{aligned}
\right.
\end{equation}
Applying Theorem \ref{theo3} we can obtain a subset of the radial spectrum, 
namely,
\begin{equation}
\left\{\mathcal{T}^\infty(\phi_k)\right\}_{k=0}^\infty\subset \sigma^{\rm 
rad}(B(o,r)).
\end{equation}
If   $\mathbb{M}^{2}_{t}=\mathbb{R}^2$,  the radial spectrum of the $B(o,1)$ is given by $\sigma^{\text{rad}}(B(o,1)=\{j_{0,k}^{2} \}_{k=1}^{\infty}$, where $j_{0,k}$ is the $k^{\text{th}}$-zero of the Bessel function $J_0$.   The following table shows the $\mathcal{T}^{j}(\phi_i)$, $j=1,2,3,9$ and $i=0,1,2$. Observe that $\mathcal{T}^{9}(\phi_1)$ agrees with $j_{0,1}^{2}$ to the $5^{\text{th}}$-decimal place.

\begin{center}
\begin{tabular}{|c|c|c|c|}\hline
 &$\phi_0=1$&$\phi_1$&$\phi_2$\\\hline
$\mathcal{T}^1$ &$5.80381$ &$31.8311$ &$85.7823$\\
$\mathcal{T}^2$ &$5.78388$ &$30.6656$ &$77.4423$\\
$\mathcal{T}^3$ &$5.78321$ &$30.5022 $ &$75.5737$\\
$\mathcal{T}^9$ &$5.78319$ &$30.4713 $ &$74.8874$\\\hline
$\lambda$&$j_{0,1}^2\approx 5.78319$&$j_{0,2}^2\approx30.4713$ 
&$j_{0,3}^2\approx 74.887$\\\hline
\end{tabular}
\end{center}

\begin{remark}This bootstrapping argument using the Green operator in Theorem 
\ref{Main2} was  discovered, and applied to a particular kind of functions,  by 
S. Sato \cite{sato} to obtain explicit estimates for the first eigenvalue of spherical 
caps of $ \mathbb{S}^{2}(1)$. It  was applied  by Barroso-Bessa \cite{barroso-bessa}  to obtain estimates for the first eigenvalue of balls in rotationally symmetric manifolds.
\end{remark}

\section{$L^1(\Omega, \mu)$-momentum spectrum}\label{sec5}Let 
 $(\Omega, ds^2, \mu)$ be a weighted bounded open subset of a Riemannian weighted Riemannian manifold with smooth boundary $\partial \Omega\neq \emptyset$. Consider a sequence of 
functions $\phi_{k}\colon \Omega \to \mathbb{R}$, $k=0,1,\ldots$ defined inductively  as the 
sequence of solutions to the following hierarchy of boundary value problems on $\Omega$. 
Let $\phi_0= 1$ in $\Omega$ and for $k\geq 1$
\begin{equation}\label{eqmomentum}\left\{
\begin{array}{rll}
\triangle_{\mu} \phi_{k} + k\phi_{k-1}&=&0 \,\,{\rm in}\,\,\, \Omega\\
\phi_{k}&=& 0 \,\,{\rm on}\,\,\partial \Omega.
\end{array}\right.
\end{equation}  The solution $\phi_1(x)$ is the mean time for the first exit of $\Omega$ of  a Brownian motion $t\to X_t$ in $\Omega$, with $X_0=x$,  see \cite{Gr}.

\begin{definition}The $L^{1}(\Omega, \mu)$-{\em momentum spectrum} of $\Omega$ is the set  
$\{\mathcal{A}_{k}(\Omega)\}_{k=1}^{\infty}$ of the  $L^{1}(\Omega, 
\mu)$-norms of $\phi_{k}$, $$\mathcal{A}_{k}(\Omega)=\int_{\Omega} \phi_k d\mu. $$ \end{definition} The number $\mathcal{A}_{1}(\Omega)$ is called the {\em torsion rigidity} of $\Omega$ because, when $\Omega \subset \mathbb{R}^{2}$, $\mathcal{A}_{1}(\Omega)$ is the torque required for a unit angle of twist per unit
length when twisting an elastic beam of uniform cross section $\Omega$, see  \cite{HMP},  \cite{markvorsen-palmer} and \cite{Mc2}. 

 Observe that if $\Omega=B(o,r)$ is a rotationally invariant geodesic ball then the solutions $\phi_{k}(x)=\phi_{k}(\vert x\vert)$ of \eqref{eqmomentum} are radial functions since $\phi_1$ is radial. In \cite{HMP}, A. Hurtado, S. Markvorsen and V. Palmer, showed that the first eigenvalue of a rotationally invariant ball can be given in terms of the momentum spectrum. They proved that  \[\lambda_{1}(B(o,r))=\lim_{k\to \infty}\frac{k\phi_{k-1}(0)}{\phi_{k}(0)}=\lim_{k\to \infty} \frac{k \mathcal{A}_{k-1}(B(o,r))}{\mathcal{A}_{k}(B(o,r))}\]and  $\phi_{\infty}(r)=\lim_{k\to \infty}\phi_{k}(r)/\phi_{k}(0)$ is a radial first eigenfunction.
In our next result we solve explicitly the 
problem \eqref{eqmomentum} in terms of the $\triangle_{\mu}$ Green operator $G$. It links the momentum spectrum with the Dirichlet spectrum of domains and in case of rotationally invariant geodesic balls and it recovers Hurtado-Markvorsen-Palmer's result.
\begin{theorem}\label{thm4.2}Let $(\Omega, ds^2, \mu)$ be a bounded weighted open subset of a Riemannian weighted Riemannian manifold 
manifold $(M,ds^{2}, \mu)$ with smooth boundary $\partial \Omega\neq 
\emptyset$. Let $\phi_{k}$ be the sequence of functions given by the boundary 
problem \eqref{eqmomentum} and $G$ the $\triangle_{\mu}$-Green operator of $\Omega$. Then 
\begin{equation}\label{eqphi_k}\phi_{k}=k!G^{k}(1).
\end{equation}Hence the exit moment spectrum is related to the Dirichlet 
spectrum as follows
\begin{equation}\begin{array}{rll}\phi_{k}(x)&=&k! \sum_{i=1}^{\infty} 
\displaystyle\frac{a_i u_i(x)}{\lambda_{i}^{k}}\\
&&\\
\mathcal{A}_k &=& k!\sum_{i=1}^{\infty}\displaystyle \frac{a_i^{2} }{\lambda_{i}^{k}},
\end{array}
\end{equation}where $a_i=\int_{\Omega} u_i(x)d\mu (x) $ and $\{ u_i\}$ is an orthonormal basis of $L^{2}(\Omega, \mu)$ formed by   eigenfunctions of $\mathcal{L}=\!-\triangle_{\mu}\vert_{ W^{2}_{0}(\Omega, \mu)}$.
\end{theorem}

\begin{corollary}\label{cor4.1}
Let $(\Omega, ds^2, \mu)$ be  weighted bounded open subset of a Riemannian weighted Riemannian manifold 
manifold $(M,ds^{2}, \mu)$ with smooth boundary $\partial \Omega\neq 
\emptyset$. Then,
\begin{equation}\label{lambda1-mom}
\lambda_1=\lim_{k\to\infty}\frac{k\mathcal{A}_{k-1}}{\mathcal{A}_{k}}\,\,\,\text{ and} \,\,\,
\lambda_2\leq \displaystyle\lim_{k\to\infty}k 
\displaystyle\frac{(k-1)\mathcal{A}_{k-2}-\lambda_1\mathcal{A}_{k-1}}{k\mathcal{
A}_{k-1}-\lambda_1\mathcal{A}_{k}}\cdot
\end{equation}
Hence,
\[
\displaystyle\frac{\lambda_2}{\lambda_1}\leq\displaystyle\lim_{k\to\infty}
\displaystyle\frac{\displaystyle\frac{(k-1)\mathcal{A}_{k-2}}{\mathcal{A}_{k-1}}
-\lambda_1}{\displaystyle\frac{k\mathcal{A}_{k-1}}{\mathcal{A}_{k}}-\lambda_1}
<\infty.
\]
\end{corollary}

\section{$L^{2}$-expansion of the Green function $g(x,y)$}\label{sec6}
 Let $(\Omega,ds^2,  \mu)$ be a  relatively compact open  subset of a Riemannian weighted manifold $(M, ds^2, \mu)$ and let
$\sigma(\Omega, \mu)=\{\lambda_{k}\}_{k=1}^{\infty}\subset [0, \infty)$ be the spectrum  of $\mathcal{L}=-\triangle_{\mu}\vert_{ W^{2}_{0}(\Omega, \mu)}$. repeated accordingly to their multiplicity.  Let   $\{u_k\}_{k=1}^{\infty}$ be an orthonormal basis in 
$L^{2}(\Omega,\mu)$ such that each function $u_k$ is an eigenfunction of 
$\mathcal{L}$ with eigenvalue $\lambda_{k}$. In   this basis, the heat kernel 
$p_{t}^{\Omega}(x,y)$  of the operator $\mathcal{L}$ admits an expansion 
\begin{equation}\label{eqheatexp}
p_{t}^{\Omega}(x,y)=\sum_{k=1}^{\infty}e^{-\lambda_{k}t}u_{k}(x)u_{k}(y).
\end{equation}This series  converges absolutely and uniformly in the domain 
$t\geq \epsilon$, $x,y\in \Omega$ for any $\epsilon >0$ as well as in the 
topology of $C^{\infty}(\mathbb{R}_{+}\times \Omega \times \Omega)$, see the 
details in \cite[Thm. 10.13]{grigoryan-book}. If we could integrate  the heat kernel expansion \eqref{eqheatexp} in the variable $t$  we would obtain  that
 \begin{equation}\label{eqgreenexpansion}g^{\Omega}(x,y)=\sum_{k=1}^{\infty}\frac{u_k(x)\cdot u_k(y)}{\lambda_k(\Omega)}\cdot\end{equation}  It is known that  identity \eqref{eqgreenexpansion} holds in the sense of distributions, see \cite[exercise 13.14]{grigoryan-book} and we may ask whether  this identity  holds in stronger topology, for instance in $L^{1}(\Omega\times \Omega, \mu)$ or $L^{2}(\Omega\times \Omega, \mu)$? It is shown in 
\cite[Exercise 13.7]{grigoryan-book} that $g^{\Omega}(x,y)\in L^{1}(\Omega\times \Omega, 
\mu)$, however, $g^{\Omega}(x,\cdot)$ does not need to belong to $L^{2}(\Omega, 
\mu)$. Indeed, let  $\Omega =B_{\mathbb{R}^{4}}(1)\subset 
\mathbb{R}^{4}$ to be the geodesic ball  of  the $4$-dimensional Euclidean space $\mathbb{R}^{4}$, with 
its canonical metric,  of radius $r=1$ and center at the 
origin $0$. The Green function of $\Omega$ is given by $$g^{\Omega}(0, 
y)=\frac{1}{2\omega_{4}}\left(\frac{1}{\vert y\vert^{2}}-1\right),$$ which obviously does not belong to $L^{2}(B_{\mathbb{R}^{4}}(1))$. Here $\omega_{4}$ is the volume of the unit $3$-sphere in  
$\mathbb{R}^{4}$. On the other hand, if  $\Omega =B_{\mathbb{R}}(1) $ is 
the geodesic ball of radius  $r=1$ in the real line $\mathbb{R}$ then the Green 
function   is $ g^{\Omega}(x,y)=( \vert x-y\vert -x\cdot y +1)/2$ which clear 
is in $L^{2}(\Omega\times \Omega)$.   
We start   showing that a necessary and sufficient condition  to $g^{\Omega}\in 
L^{2}(\Omega\times \Omega, \mu)$, for any measure $\mu$, is that   ${\rm dim}(M)=1,2,3$ and in these dimensions  identity \eqref{eqgreenexpansion} holds in $L^{2}(\Omega\times \Omega, \mu)$. 
 Precisely, we have the following result.
\begin{theorem}\label{ThmMainA}Let $(M,ds^2, \mu)$ be a weighted 
Riemannian manifold and let $\Omega\subset M$ be a bounded open 
subset of $M$ with smooth boundary $\partial \Omega \neq \emptyset$. Then the 
Green function $g^{\Omega}\in L^{2}(\Omega \times \Omega, \mu)$ if and only if 
${\rm dim}(M)=1,2,3.$ Furthermore,
\begin{equation}g^{\Omega}(x,y)=\sum_{k=1}^{\infty} \frac{u_{k}(x)\cdot 
u_{k}(y)}{\lambda_{k}(\Omega)},\label{eqGreenSeries}
\end{equation} where the series converges in $L^{2}(\Omega \times \Omega, 
\mu)$,  and   
$\{u_i\}_{k=1}^{\infty}$ is an orthonormal  basis formed by    eigenfunctions 
associated to the eigenvalues $\{\lambda_{k}(\Omega)\}_{k=1}^{\infty}$ of 
$\Omega$. Moreover \begin{equation}\label{eq5.3}\Vert g^{\Omega}\Vert_{L^{2}(\Omega \times \Omega)} = 
\sum_{k=1}^{\infty}\displaystyle \frac{1}{\lambda_{k}^{2}(\Omega)}\cdot\end{equation}
If $n=1$ then 
\begin{equation}
\int_{\Omega}g^{\Omega}(x,x)d\mu(x)=\sum_{k=1}^{\infty} 
\frac{1}{\lambda_{k}(\Omega)}\cdot
\end{equation}
\end{theorem}It should be remarked  that this result above is a direct consequence of Weyl's assymptotic formula for the eigenvalues $\lambda_{i}(\Omega)$. The first consequence of Theorem \ref{ThmMainA}  regards a result due to Gr\"{u}ter and Widman,  see \cite[Thm. 1.1]{gruter-widman}. They  proved that for any 
$y\in \Omega \subset \mathbb{R}^{n}$, $n\geq 3$,  the  function  \[x\to 
g^{\Omega}(x,y)\in L^{\ast}_{\frac{n}{n-2}}(\Omega),\] where
$L_{p}^{\ast}(\Omega)$ is defined by \[L_{p}^{\ast}(\Omega)\colon\!\!=\{ 
f\colon \Omega \to \mathbb{R}\cup\{\infty\}, f\,\,\mathrm{measurable}\,\, 
\mathrm{and}\,\, \Vert f\Vert_{L_{p}^{\ast}(\Omega)}<\infty\}\]with
\[\Vert f\Vert_{L_{p}^{\ast}(\Omega)}\colon\!\!=\sup_{t>0} t \cdot \nu (\{x\in 
\Omega\colon \vert f(x)\vert >t\})^{1/p}.\] 
It is easy to see   that $L^{p}(\Omega)\subset L_{p}^{\ast}(\Omega)$, thus
coupling  Theorem 
\ref{ThmMainA}  with   Gr\"{u}ter and Widman's result we obtain a  more precise 
statement. 
\begin{corollary}Let $\Omega \subset \mathbb{R}^{n}$, $n\geq 3$ be a bounded 
open subset with smooth boundary. Let $g^{\Omega}$ its Green function 
$($associated to the Laplacian $\triangle$\! $)$.  
\begin{itemize}\item If $n=3$ then $x\to g^{\Omega}(x,y)\in L^2(\Omega)\cap 
L_{3}^{\ast}(\Omega)$,  $y\in \Omega$. 
 \item If $n\geq 4$ then  $x\to g^{\Omega}(x,y)\in  
L_{\frac{n}{n-2}}^{\ast}(\Omega)\setminus L^2(\Omega) $, $y\in \Omega$.
\end{itemize}
\end{corollary}

\begin{remark}In general,  $L_{3}^{\ast}(\Omega)\setminus 
L^{2}(\Omega)\neq\emptyset$ and $L^{2}(\Omega)\setminus 
L_{3}^{\ast}(\Omega)\neq\emptyset$. Indeed,   letting $\Omega = B_{r}(o)\subset 
\mathbb{R}^{3}$ and $f_{\alpha}(x)=1/\vert x\vert^{\alpha}$ we have that 
$f_\alpha \in L_{3}^{\ast}(\Omega)\setminus L^2(\Omega)$ if $\alpha >1$ and  
$f_\alpha \in L^{2}(\Omega)\setminus  L_{3}^{\ast}(\Omega)$ if $0<\alpha <1/2$. 
\end{remark}

\section{Proofs of the results}\label{sec7}The structure of this section is the following. First, we  will prove Theorem \ref{ThmMainA} and  then Theorem \ref{thmMark}. Then we will prove  Theorems \ref{thmMain1-intro} and \ref{thmMain2}. Finally we present the proofs of Theorems \ref{Main2} \& \ref{thm4.2}.

\subsubsection{Proof of Theorem \ref{ThmMainA}} 
Let $(\Omega, ds^2, \mu)$ be a bounded   open subset of a  weighted $m$-manifold $(M,ds^{2}, \mu)$ with smooth 
boundary $\partial \Omega \neq \emptyset$. Let 
$\sigma(\Omega)=\{\lambda_{k}(\Omega)\}_{k=1}^{\infty}$ be set of eigenvalues of   
$\mathcal{L}=-\triangle_{\mu}\vert_{W_{0}^{2}(\Omega, \mu)}$, repeated 
according to multiplicity. One has that the Weyl's asymptotic formula holds for the 
eigenvalues $\lambda_{k}$, 
\begin{equation}
\lambda_{k}(\Omega)\approx c_{m}\cdot\left(\frac{k}{\mu 
(\Omega)}\right)^{2/m}\,\,{\rm as}\,\, k\to \infty,
\end{equation}where $c_{m}>0$ is the same constant as in $\mathbb{R}^{m}$, 
depending only on the dimension $m={\rm dim}(M)$, see 
\cite[p.7]{GriHeatDen}. Thus 
$$
\sum_{i=1}^\infty\frac{1}{\lambda_k^2(\Omega)}<\infty \Leftrightarrow m=1,2,3
$$
and 
$$
\sum_{i=1}^\infty\frac{1}{\lambda_k(\Omega)}<\infty\Leftrightarrow m=1.
$$
Let
 $\{g_k\colon \Omega \times \Omega \to \mathbb{R}\}\subset  L^{2}(\Omega\times 
\Omega, \mu)$ be a sequence of functions  defined by  
\[g_k(x,y)=\sum_{i=1}^{k}\frac{u_i(x)u_i(y)}{\lambda_{i}(\Omega)}\cdot\] Let  $k_2>k_1$ and 
compute \[\Vert g_{k_2}-g_{k_1}\Vert_{L^{2}(\Omega\times \Omega, \mu)}^{2}
= \sum_{i=k_{1}+1}^{k_2}\displaystyle\frac{1}{\lambda_{i}^{2}(\Omega)}.\] The sequence $\{g_k\}$ is a Cauchy sequence 
 in $L^{2}(\Omega\times \Omega,\mu)$ iff $\sum_{i=1}^{\infty}\frac{1}{\lambda_{i}^{2}(\Omega)}<\infty$. Then, $g_k\to g_{\infty}$ in $L^{2}(\Omega\times \Omega, \mu) $ if and only if   ${\rm dim}(M)=1,2,3$, where
\[g_{\infty}(x,y)=\sum_{i=1}^{\infty}\frac{u_i(x)u_i(y)}{\lambda_{i}(\Omega)}\in L^{2}(\Omega\times \Omega, \mu).\]  

On the other hand,  it is known  that the Green function $g^{\Omega}$  satisfies the functional identity \[ 
g^{\Omega}(x,y)=\sum_{i=1}^{\infty}\frac{u_i(x)u_i(y)}{\lambda_{i}(\Omega)}
 \]in the sense of distributions, see  \cite[p. 
348]{grigoryan-book}. Therefore,  $g^{\Omega}=g_{\infty}\in L^{2}(\Omega\times \Omega, \mu)$ if and only if $m=1,2,3$ 
and \[\Vert g^{\Omega}\Vert_{L^{2}(\Omega\times \Omega, 
\mu)}^{2}=\sum_{i=1}^\infty\frac{1}{\lambda_i^2(\Omega)}<\infty.\]If $m=1$, then 
$g^{\Omega}(x,x)=\sum_{i=1}^{\infty}\frac{u_i^{2}(x)}{\lambda_{i}(\Omega)}$ and 
\[\int_{\Omega}g^{\Omega}(x,x)d\mu (x) = \sum_{i=1}^{\infty} 
\frac{1}{\lambda_i(\Omega)}\cdot\]
\subsubsection{Proof of Theorem \ref{thmMark}}Let $\varphi \colon M\to N$ be a properly immersed  $m$-submanifold $M$ into of a Riemannian 
manifold $N$. Let $t_{N}(x)={\rm dist}_{N}(p,x)$ be the distance in $N$ from 
a fixed point $p=\varphi(q)\in N$. Let $\Omega_r =
\varphi^{-1}(B_{N}(p,r))$ be an extrinsic ball that contains $q$, where $ 
B_{N}(p,r)$ is the geodesic ball of $N$ with center at $p$ and radius $r< \min 
\{ {\rm inj}(p), \pi/\sqrt{k}\}$, $k=\sup K_{N} $ and where  
$\pi/\sqrt{k}=\infty$ if $k\leq 0$.

 Let $E_r(x)$ be the mean time of the first 
exit from $\Omega_r$ for a Brownian motion particle starting at $x\in 
\Omega_r$. A fundamental observation of Dynkin \cite[vol.II, p.51]{Dy} states 
that the function $E_r$ satisfies the Poisson equation 
\begin{equation}\left\{\begin{array}{rrl}\triangle E_r&=&\!\!\!-1 \,\,{\rm in}\,\, \Omega_r\\
E_r&=&\,0 \,\,\,{\rm on}\,\,\partial \Omega_{r}\end{array}\right.\label{eqdynkin}
\end{equation}If $g^{\Omega_r}(x,y)=g(x,y) $ is the Green function of $\Omega_r$, with Dirichlet 
boundary data, then \[E_r(x)=\int_{\Omega_r}g(x,y)d\nu(y).\] Applying 
Cauchy-Schwarz and assuming that $m=2,3$ we obtain
\begin{eqnarray}\label{eq2.5}
\int_{\Omega_r}E_{r}^{2}(x)d\nu(x) 
&=&\int_{\Omega_r}\left(\int_{\Omega_r}g(x,y)d\nu(y)\right)^2d\nu(x)\nonumber \\
&\leq &{\rm 
vol}(\Omega_r)\cdot\int_{\Omega_r}\int_{\Omega_r}g^2(x,y)d\nu(y)d\nu(x)\\
&=&{\rm 
vol}(\Omega_r)\sum_{k=1}^\infty\frac{1}{\lambda_k(\Omega_r)^2}\cdot\nonumber
\end{eqnarray} 
Assume that $N=\mathbb{R}^{n}$, $p=o\in 
\mathbb{R}^{n}$. Let $\mathbf{\widetilde{E}}_r\colon B^{m}(0,r)\to \mathbb{R} $ be the mean exit time of the first 
exit from the geodesic ball $B^{m}(o,r)\subset \mathbb{R}^{m}.$ It 
is  known that  $\mathbf{\widetilde{E}}_r$ is radial, i.e.
 $\mathbf{\widetilde{E}}_r(y)=\mathbf{\widetilde{E}}_r(t (y))$, $t(y)=\vert y-o\vert$, $y\in \mathbb{R}^{m}$. Denote  by $\widetilde{E}_r $ the 
transplant of $\mathbf{\widetilde{E}}_r$ to $B^n(o,r)$, i.e.,  the 
 function  $\widetilde{E}_r \colon B^n(o,r)\subset \mathbb{R}^{n}\to \mathbb{R}$ 
defined by $\widetilde{E}_r(z)=\mathbf{\widetilde{E}}_r(t(z))$.
 Consider the restriction of $\widetilde{E}_r(z)$ to the immersion $\varphi(M)$, i.e. $x\in \Omega_r\to\widetilde{E}_{r}(\varphi (x))$.
In \cite{Mar}, Steen  Markvorsen proved that if the immersion  $\varphi \colon 
M\to \mathbb{R}^{m+1}$ is a minimal hypersurface then 
$E_r(x)=\widetilde{E}_{r}(\varphi (x))=\mathbf{\widetilde{E}}_{r}(t(\varphi (x)))$. 
Solving  problem \eqref{eqdynkin}, we have 
\begin{equation}\label{eq2.6}
\mathbf{\widetilde{E}}_{r}(s)=\int_s^r\frac{{\rm vol}(B^m(o,\zeta))}{{\rm 
vol}(\partial B^m(o,\zeta))}d\zeta=\frac{1}{2m}(r^2-s^2),
\end{equation}therefore, by inequality \eqref{eq2.5}, we have
\begin{equation*}
\begin{aligned}
{\rm vol}(\Omega_r)\sum_{k=1}^\infty\frac{1}{\lambda_k(\Omega_r)^2}\geq& 
\frac{1}{4m^2}\int_{\Omega_r}\left(r^4+t^4(x)-2r^2t^2(x)\right)d\nu(x), 
\end{aligned}
\end{equation*}where we identified $t(\varphi(x))=t(x)$. 
Applying co-area formula for the extrinsic distance function 
$t:M\to\mathbb{R}_+$ we obtain
\begin{equation*}
\begin{aligned}
\int_{\Omega_r}\left(r^4+t^4(x)-2r^2t^2(x)\right)d\nu(x)=&\int_0^r\left(\int_{
\partial \Omega_s}\frac{r^4+t^4(x)-2r^2t^2(x)}{\vert\nabla\, 
t\vert}dA\right)ds\\
=&\int_0^r(r^4+s^4-2r^2s^2)\left(\int_{\partial \Omega_s}\frac{1}{\vert\nabla\, 
t\vert}dA\right)ds\\
\geq& \int_0^r(r^4+s^4-2r^2s^2){\rm vol}(\partial \Omega_s)ds.
\end{aligned}
\end{equation*}
On the other hand,  we have that ${\rm vol}(\partial \Omega_s)\geq 
m\omega_ms^{m-1}$,  see \cite{Palmer}. Then
\begin{equation}
\begin{aligned}
\int_{\Omega_r}\left(r^4+t^4(x)-2r^2t^2(x)\right)dV(x)\geq&\, 
m\omega_m\int_0^r(r^4+s^4-2r^2s^2)s^{m-1}ds\\
=&\,m\omega_mr^{4+m}\left(\frac{1}{m}+\frac{1}{4+m}-\frac{2}{2+m}\right)\cdot
\end{aligned}
\end{equation}
In order to simplify the notation let us denote by 
$A_m:=1+\frac{m}{4+m}-\frac{2m}{2+m}$. Hence,
\begin{equation}\label{eq2.61}
\begin{aligned}
{\rm vol}(\Omega_r)\cdot \sum_{k=1}^\infty\frac{1}{\lambda_k(\Omega_r)^2}\geq& 
\frac{r^4}{4m^2}A_m\omega_mr^m.
\end{aligned}
\end{equation}
That proves the lower bound in \eqref{eq2.3Gimeno}. To prove the upper bound recall that Cheng, Li and 
Yau proved in \cite{CLY} that
\begin{equation}
\lambda_k(\Omega_r)\geq 4\pi\left(\frac{k}{e}\right)^{2/m}\frac{1}{{\rm 
vol}(\Omega_r)^{2/m}}\cdot
\end{equation}
Therefore,
\begin{eqnarray}\label{eq2.8}
\sum_
{k=1}^\infty\frac{1}{\lambda_k(\Omega_r)^2}&\leq &\frac{e^{4/m}}{16\pi^2}{\rm 
vol}(\Omega_r)^{4/m}\sum_{k=1}^\infty\frac{1}{k^{4/m}}\nonumber \\
&& \\
&=&\frac{e^{4/m}}{16\pi^2}{\rm vol}(\Omega_r)^{4/m}\zeta(4/m).\nonumber
\end{eqnarray} Observe that $\zeta(2)=\pi^{2}/6$. Putting together  inequalities \eqref{eq2.61} and \eqref{eq2.8} we obtain \[
A_m \cdot\omega_m\cdot\left(\frac{r^m}{{\rm vol}(\Omega_r)}\right)\cdot r^4\leq 
\sum_{k=1}^\infty\frac{1}{\lambda_k^2(\Omega_r)}\leq B_m\cdot\zeta(4/m)\cdot \left(\frac{{\rm 
vol}(\Omega_r)}{r^m}\right)^{4/m}\cdot r^4.
\]
In order to obtain inequality (\ref{eq:2.11}) we have by the monotonicity formula, see  \cite{MP-2012,Palmer}, that the function
$$
r\to \frac{{\rm vol}(\Omega_r)}{\omega_mr^m}
$$
is an increasing function. Moreover by the classical results of Jorge-Meeks in  \cite{JM}, see also \cite{anderson,Che4}, 
$$
\lim_{r\to\infty}\frac{{\rm vol}(\Omega_r)}{\omega_mr^m}=\mathcal{E}.
$$
Therefore
$$
{\rm vol}(\Omega_r)\leq \omega_m \mathcal{E}r^m
$$
and the theorem follows taking in consideration that \[ \mathcal{E}=\left\{\begin{array}{rll}\sum_{i=1}^{k}I_i& {\rm if}& m=2\\
k & {\rm if}& m=3,\end{array}\right.\] where $I_i$ is the geometric index of the end $E_i$, see details in  \cite{JM}.

\subsubsection{Proof of Theorem \ref{thmMain1-intro}}The metric on the geodesic ball 
$B(o,r)\subset \mathbb{M}_{h}^{n}$ is expressed, in polar coordinates, as 
$ds^2 =dt^{2}+h^{2}(t)d\theta^{2}$.
The Laplacian $\triangle$ of this metric is given  by \begin{eqnarray}\triangle (t,\theta)& = & 
\frac{\partial^{2}}{\partial t^{2}} + 
(n-1)\frac{h'}{h}(t)\frac{\partial}{\partial t} + 
\frac{1}{h^2}(t)\triangle_{\theta}\nonumber \\
&=& L_0 + \frac{1}{h^2}(t)\triangle_{\theta}\nonumber ,\end{eqnarray}
and $\triangle_{\theta}$ is the Laplacian on 
$\mathbb{S}^{n-1}$. Observe that the radial eigenvalues of $B(o,r)$ are the eigenvalues of the operator $L_0$ in the following eigenvalue problem.
%
%
%
%
\begin{equation}\label{eq2.19}\left\{\begin{array}{rll}
L_0 u+ \lambda u & = &0 \\
 u'(0)&=&0 \\
 u(r)&=&0.
\end{array} \right.
\end{equation}In order to study this eigenvalue problem, define the following  space of functions
\begin{equation}
\Lambda:=\left\{u\in W^{2}([0,r],\mu)\, :\,\lim_{t\to 0^+}u'(t)=0 \,\text{ and } u(0)=0\right\}
\end{equation}
 and the density  in $[0,r]$ given by $d\mu(t)=\omega_nh^{n-1}(t)dt$. Observe that
$f\in L^2([0,r],\mu)$ if and only if $t\circ f\in L^2(B(o,r))$. Define the bilinear form $\mathcal{E}_{bf}$ acting on Lipschitz functions \[\mathcal{E}_{bf}(f,g)=\int_{0}^{r} f'(s)g'(s)ds\] and let
$\mathcal{F}$ be the closure of $\Lambda$ in $L^{2}([0,r],\mu)$ with respect to the norm
\begin{equation}
\Vert f\Vert_{\mathcal{F}}^{2}=\Vert f\Vert_{L^{2}([0,r], \mu)}^{2}+ \mathcal{E}_{bf}(f,f)
\end{equation}
 The bilinear  form $\mathcal{E}_{bf}$  acting on $\mathcal{F}$, in the distributional sense, is a Dirichlet form, i.e. it  has the following properties.
\begin{enumerate}
\item \emph{Positivity}:  $\mathcal{E}_{bf}(f)=\mathcal{E}_{bf}(f,f)\geq 0$ for any $f\in \mathcal{F}$.
\item \emph{Closedness}: the space $\mathcal{F}$ is a Hilbert space with respect to the following product
\[\langle f, g\rangle=
(f,g)+\mathcal{E}_{bf}(f,g).
\]
\item \emph{The Markov property}: if $f\in \mathcal{F}$ then the function
\[
g:=\min\{1,\max\{f,0\}\}
\]
also belongs to $\mathcal{F}$ and $\mathcal{E}_{bf}(g)\leq \mathcal{E}_{bf}(f)$. Here we  used the  shorthand notation $\mathcal{E}_{bf}(f):=\mathcal{E}_{bf}(f,f)$.
\end{enumerate}Any Dirichlet form $\mathcal{E}_{bf}$  has a {\em generator} $\mathcal{L}$ which is a non-positive definite self-adjoint operator on $L^{2}([0,r], \mu)$ with domains $\mathcal{D}=\mathcal{D}(\mathcal{L})\subset \mathcal{F}$ such that 
\begin{equation}
\mathcal{E}_{bf}(f,g)=(-\mathcal{L}(f), g)
\end{equation} for $f\in \mathcal{D}$ and $g\in \mathcal{F}$ where $\mathcal{D}$ is dense in $\mathcal{F}$, see details in \cite[Sec.2.2]{Gri09}.

\begin{proposition}The operator
$\mathcal{L}\vert_{ \mathcal{D}}$ is an extension of $L_{0}\vert_{ \Lambda}$, this is,
\begin{equation}
\mathcal{L} f=L_0f, \,\,{\rm for}\,\,{\rm any}\,\,f \in \Lambda.
\end{equation}
\end{proposition}
{\noindent Proof:}
For any $f,g\in \Lambda$
\begin{eqnarray}
(\mathcal{L}f,g)&=&-\mathcal{E}_{bf}(f,g)\nonumber \\
&=&-\omega_n\int_0^rf'(t)g'(t)h^{n-1}(t)dt\nonumber\\
&=&-\omega_n\int_0^r\left[\frac{d}{dt}\left(f'(t)g(t)h^{n-1}(t)\right)-g(t)\frac{d}{dt}\left(f'(t)h^{n-1}(t)\right)\right]dt\\
&=&-\omega_n\left[f'(t)g(t)h^{n-1}(t)\right]_0^r+\int_0^rg(t)L_0 f(t)d\mu(t)\nonumber\\
&=&(L_0f,g).\nonumber
\end{eqnarray} The generator $\mathcal{L}$ determines the heat semigroup $P_t=e^{-\mathcal{L}t},\,\,t\geq 0$ which posses a heat kernel $p(t,x,y)$ and Green function $g(x,y)=\int_{0}^{\infty}p(t,x,y)dt$. Observe that $\mathcal{L}\vert_{ \mathcal{D}}$ is a self-adjoint extension of $L_0\vert_{ \Lambda}$.  Thus, the solution of  eigenvalue problem \eqref{eq2.19} is an infinite sequence of eigenvalues $0< \lambda_{1}^{\rm rad} < \lambda_{2}^{\rm rad} < \cdots$, (the radial spectrum of $B(o,r)$).  Moreover, $L_0=\triangle_{\mu}$, then we have   by Theorem \ref{ThmMainA} that \begin{equation}\sum_{i=1}^{\infty}\frac{1}{\lambda_i^{\rm rad}(B(o,r))} = \int_{0}^{r} g(x,x)dx. \label{eq6.15}\end{equation}
We need to determine  the Green function $g(x,y)$ for the operator $\mathcal{L}$.
\begin{proposition}\label{prop-green-rad}The Green function $g(x,y)$ for the operator $\mathcal{L}$ is given by
\begin{equation}\label{green-rad1}
g(x,y)=\int_x^r \frac{1}{\omega_nh^{n-1}(t)}\theta_y(t)dt,
\end{equation}
where
\begin{equation}
\theta_y(t):=\begin{cases}
1\quad \text{ if }\quad t\geq y\\
0\quad \text{ if }\quad t<y.
\end{cases}
\end{equation}
Moreover,
\begin{equation}
G(f)(x)=\int_0^rg(x,y)f(y)d\mu(y)=T(f)(x).
\end{equation}Here $T$ is the operator defined in \eqref{T}.
\end{proposition}  
\noindent{Proof:}
We need   to prove that $g(r,y)=0$, $\lim_{x\to 0^+}\displaystyle\frac{\partial}{\partial x}g(x,y)=0$, $\mathcal{L}_{x}(x,y)=0$ for $x\neq y$ and
 $\mathcal{L}g(x,\,\cdot)=-\delta_x$.
The two first properties are straight forward from equation (\ref{green-rad1}). When $x\neq y$ we have 
\begin{equation}
g(x,y)=\left\{\begin{array}{lll}
\displaystyle\int_x^r \frac{1}{\omega_nh^{n-1}(t)}dt&\text{ if }&  x\geq y\\
&&\\
\displaystyle\int_y^r \frac{1}{\omega_nh^{n-1}(t)}dt& \text{ if }& x\leq y.
\end{array}\right.
\end{equation}  
Then,
\begin{equation}
\mathcal{L}_xg(x,y)=\left\{\begin{array}{cll}
{L_0}\left(\displaystyle\int_x^r \frac{1}{\omega_nh^{n-1}(t)}dt\right)=0& \text{ if }& x\geq y\\
0&\text{ if }& y\leq x.
\end{array}\right.
\end{equation}The last requirement
$\mathcal{L}g(x,\,\cdot)=-\delta_x$ is proven  now.  We have to show that if $f\in \mathcal{F}$ then 
\begin{eqnarray}
 -\int_0^r\mathcal{L}_xg(x,y)f(y)d\mu(y)
&=&-\mathcal{L}_x\int_0^rg(x,y)f(y)d\mu(y)\\
&=&-\mathcal{L}_xG(f)(x)\nonumber \\
&=&f(x).\nonumber
\end{eqnarray}On the other hand,
\begin{eqnarray}G(f)(x)&=&\int_0^rg(x,y)f(y)d\mu(y)\nonumber \\
&=&\int_0^r\left(\int_x^r\frac{\theta_y(t)}{h^{n-1}(t)}dt\right)f(y)h^{n-1}(y)dy \nonumber
\\
&=&\int_x^r\frac{1}{h^{n-1}(t)}\left(\int_0^t\theta_y(t)f(y)h^{n-1}(y)dy+\int_t^r\theta_y(t)f(y)h^{n-1}(y)dy\right)dt \nonumber\\
&=&\int_x^r\frac{1}{h^{n-1}(t)}\left(\int_0^tf(y)h^{n-1}(y)dy\right)dt\nonumber\\
&=&T(f)(x).\nonumber
\end{eqnarray}Hence we have to prove that $\mathcal{L}\circ T=-\text{id}$. Since $T\colon\Lambda\to\Lambda$, for any $u\in \Lambda$

\begin{eqnarray}\mathcal{L} \circ T (u)&=&D_h^2\circ T (u)=\frac{\partial^{2}T(u)}{\partial t^{2}} + (n-1)\frac{h'}{h}(t)\frac{\partial T(u)}{\partial t}\cdot\label{eq10}  \end{eqnarray}However, 
 \[\begin{array}{rll}\displaystyle\frac{\partial^{2}T(u)}{\partial t^{2}}(t)&=&(n-1)\displaystyle \frac{h'(t)}{h(t)}\frac{1}{h^{n-1}}\int_{0}^{t}h^{n-1}(s)u(s)ds - u(t)\\
 &&\\
 \displaystyle(n-1)\frac{h'(t)}{h(t)}\cdot \frac{\partial T(u)}{\partial t}(t)&=&-\displaystyle (n-1)\frac{h'(t)}{h(t)}\frac{1}{h^{n-1}(t)}\int_{0}^{t}h^{n-1}(s)u(s)ds.
 \end{array}\]Then
 $\mathcal{L} \circ T (u)(x)=-u(x)$
and the proposition follows.
To prove Theorem \ref{thmMain1-intro} we have from \ref{eq6.15} that  \begin{eqnarray}\sum_{i=1}^{\infty}\frac{1}{\lambda_i^{\rm rad}(B(o,r))}&= &\int_{0}^{r} g(x,x)d\mu (x)\nonumber \\
&=& \int_{0}^{r}\left( \int_{x}^{r} \frac{dt}{h^{n-1}(t)}dt\right) h^{n-1}(x)dx\nonumber \\
&=& \int_0^r\frac{\int_0^xh^{n-1}(t)dt}{h^{n-1}(x)}dx\\
&=& \int_{0}^{r} \frac{V(s)}{S(s)}ds,\nonumber\label{eq6.23}
\end{eqnarray}This proves identity \eqref{eqMain1-intro}. If 
$\mathbb{M}_{h}^{n}$ is stochastically incomplete, then  its spectrum is discrete, say $\sigma (\mathbb{M}_{h}^{n})=\{ \lambda_1(\mathbb{M}_{h}^{n}) < \lambda_2(\mathbb{M}_{h}^{n}) \leq \cdots\}$. Taking the limits in \eqref{eq6.23} we obtain \[ \lim_{r\to \infty} \sum_{i=1}^{\infty}\frac{1}{\lambda_i^{\rm rad}(B(o,r))}=\int_{0}^\infty \frac{V(s)}{S(s)}ds<\infty.\]  To prove identity \eqref{eqMain2-intro} we  recall that $\lambda^{\rm rad}_{i}(\mathbb{M}_{h}^{m})=\lim_{r\to \infty}\lambda^{\rm rad}_{i}(B(o,r))$. 
%
%
This proves that
  \[ \sum_{i=1}^{\infty}\frac{1}{\lambda_i^{\rm rad}(\mathbb{M}^{n}_{h})}=\int_{0}^\infty \frac{V(s)}{S(s)}ds<\infty.\]  

\subsubsection{Proof of Theorem \ref{thmMain2}}The proof of Theorem \ref{thmMain2} is closed to the proof of Theorem \ref{thmMain1-intro}.  The spectrum of  $B(o,r)$, without repetitions, is the union of   the $\nu_{l}$-spectrums $\sigma^{l}(B(o,r))$, $l=0, 1, \ldots$ \[\sigma (B(o,r))=\cup_{l=0}^{\infty} \sigma^{l}(B(o,r))=\{ \lambda_{l,j}\}_{l=0,j=1}^{\infty,\,\, \infty},\]each $\lambda_{l,i}$ with multiplicity $\delta(l,m)$. The eigenvalues of the $l$-spectrum $\sigma^{l}(B(o,r))$, $l\geq 1$,  are the eigenvalues of the the operator
\[L_l(T)(t)=T''(t)+(n-1) \displaystyle\frac{1}{t}T'(t) - \displaystyle\frac{\nu_l}{t^2}T(t),\] $\nu_l=l(l+m-2)$, in the following Dirichlet eigenvalue problem on $[0,r]$. 
\begin{equation}
\begin{array}{lll}
T''+(m-1) \displaystyle\frac{1}{t}T' + (\lambda - \displaystyle\frac{\nu_l}{t^2})T=0&{\rm in}& [0,r] \end{array}
\end{equation} with initial conditions
 $ T(t)\sim c\cdot t^l$  as $ t\to 0$  when $l=1,2\ldots$ and $T(r)=0$. The procedure to show that the operator $L_l$ has a self-adjoint extension $\mathcal{L}_{l}$ and a Green function $g_{_l}$ is similar to the procedure in the proof of Theorem \ref{thmMain1-intro}. By Theorem \ref{ThmMainA} we have that 
 \begin{equation}\sum_{i=1}^{\infty}\frac{1}{\lambda_{l,i}(B(o,r))} = \int_{0}^{r} g_{_l}(x,x)dx. \label{eq6.151}\end{equation} We need to find the Green function $g_{_l}$.   \begin{proposition}The Green function $g_{_l} (x,y)$ for the $\mathcal{L}_l$ operator on  $M=[0,r]$  with density $d\mu(x)=\omega_nx^{n-1}dx$ boundary conditions
$$
u'(0)=u(r)=0,\quad {\rm with }\, \,u(x)=g(x,y)\,{\rm for}\,{\rm any }\,\,  y\in (0,r)
$$ 
is given by
\begin{equation}
g(x,y)=\left\{
\begin{array}{lcc}
\displaystyle\frac{x^{l}y^{\alpha}}{\beta\omega_ny^{n-1}}\left(\frac{1}{y^\beta}-\frac{1}{r^\beta}\right)&{\rm if }&0\leq x<y\\
&\\
\displaystyle\frac{x^{l}y^{\alpha}}{\beta\omega_ny^{n-1}}\left(\frac{1}{x^\beta}-\frac{1}{r^\beta}\right)&{\rm if }&y\leq x\leq r\\
\end{array}\right.
\end{equation}
with $\alpha=l+n-1$ and $\beta=2l+n-2$.
\end{proposition}
\begin{proof}
Observe first of all that
$$
u'(0)=\left.\frac{\partial}{\partial x}g(x,y)\right\vert_{x=0}=0,\quad u(r)=g(r,y)=0,\quad\forall y\in [0,r].
$$
On the other hand 
$$
\mathcal{L}_l(u(x))=u''(x)+(n-1)\frac{u'(x)}{x}-\frac{l(l+n-2)u(x)}{x^2}=0
$$
because
$$
\mathcal{L}_l(x^l)=\mathcal{L}_l(x^{l-\beta})=0.
$$
For any function $f\colon[0,r]\to \mathbb{R}$, we have then
\begin{eqnarray}
G_{_l}(f)(x)&=&\int_0^rg_{_l}(x,y)f(y)d\mu(y)\nonumber \\
&& \nonumber \\
&=&\int_0^xg_{_l}(x,y)f(y)d\mu(y)+\int_x^rg_{_l}(x,y)f(y)d\mu(y)\nonumber \\
&& \nonumber \\
&=&\frac{x^l}{\beta}\left(\frac{1}{x^\beta}-\frac{1}{r^\beta}\right)\int_0^xy^\alpha f(y)dy+\frac{x^l}{\beta}\int_x^ry^\alpha\left(\frac{1}{y^\beta}-\frac{1}{r^\beta}\right)f(y)dy \nonumber\\
&& \nonumber \\
&=&\frac{x^{l-\beta}}{\beta}\int_0^xy^\alpha f(y)dy+\frac{x^l}{\beta}\int_x^ry^{\alpha-\beta} f(y)dy-\frac{x^l}{\beta r^\beta}\int_0^ry^\alpha f(y)dy.\nonumber\\
&& \nonumber \\
{\rm and}&&\nonumber\\
&& \nonumber \\
\mathcal{L}_l\left(G_{_l}(f)(x)\right)&=&\mathcal{L}_l\left(\frac{x^{l-\beta}}{\beta}\int_0^xy^\alpha f(y)dy\right)+\mathcal{L}_l\left(\frac{x^l}{\beta}\int_x^ry^{\alpha-\beta} f(y)dy\right)\nonumber \\
&& \nonumber \\
&&-\mathcal{L}_l\left(\frac{x^l}{\beta r^\beta}\right)\int_0^ry^\alpha f(y)dy\nonumber \\
&& \nonumber \\
&=&\mathcal{L}_l\left(\frac{x^{l-\beta}}{\beta}\right)\int_0^xy^\alpha f(y)dy+\mathcal{L}_l\left(\frac{x^l}{\beta}\right)\int_x^ry^{\alpha-\beta} f(y)dy\nonumber \\
&& \nonumber \\
&&-x^{l+\alpha-\beta-1}f(x)\nonumber \\
&& \nonumber \\
&=&-f(x),\nonumber
\end{eqnarray}where we have applied $l+\alpha-\beta-1=0$. Therefore  we conclude that, $$
\mathcal{L}_l\circ G_{_l}(f)(x)=-f(x),
$$
and $g_{_l}$ is a Green function for our problem. Now\begin{eqnarray}
\sum_{i=1}^\infty\frac{1}{\lambda_{l,i}}&= &\int_0^rg_{_l}(x,x)d\mu(x)=\int_0^r\frac{x^{l+\alpha}}{\beta}\left(\frac{1}{x^\beta}-\frac{1}{r^\beta}\right)dx\nonumber \\
&=&\frac{r^{l+\alpha-\beta+1}}{(l+\alpha+1)(l+\alpha-\beta+1)}\nonumber \\
&=&\frac{r^2}{2(2l+n)}=\left(\frac{1}{1+2\frac{l}{n}}\right)\frac{r^2}{2n}\nonumber \\
&=&\left(\frac{1}{1+2\frac{l}{n}}\right)\max_{x\in B_{\mathbb{R}^n}(r)}E_r(x).\nonumber
\end{eqnarray}This proves \eqref{eqMain3}.
 
\noindent To prove \eqref{eqMain4} we proceed as follows.
\begin{eqnarray}
\sum_{i=1}^\infty\frac{1}{\left(\lambda_{l,i}\right)^2}&=&\int_0^r\int_0^rg_{_l}(x,y)g_{_l}(y,x)d\mu(y)d\mu(x)\nonumber\\
&=&\int_0^r\int_0^r\omega_n^2x^{n-1}y^{n-1}g_{_l}(x,y)g_{_l}(y,x)dydx\nonumber \\
&=&\int_0^r\omega_n^2x^{n-1}\left(\int_0^xg_{_l}(x,y)g(y,x)y^{n-1}dy+\int_x^rg(x,y)g_{_l}(y,x)dy\right)dx \nonumber \\
&=&\frac{1}{\beta^2}\int_0^r\!\!\left(\int_0^x\!\!x^{l+\alpha}y^{l+\alpha}\!\left(\!\frac{1}{x^\beta}\!-\!\frac{1}{r^\beta}\right)^2\!\!\!dy\!+\!\!\int_x^r\!\!\!x^{l+\alpha}y^{l+\alpha}\left(\frac{1}{y^\beta}-\frac{1}{r^\beta}\right)^{\!2}\!\!\!dy\!\right)\!dx\nonumber\\
&=& \frac{r^{2(l+\alpha-\beta+1)}}{(\alpha+l+1)^2(\alpha+l-\beta+1)(2+2\alpha-\beta+2l)}\nonumber\\
&=&\frac{r^4}{2(2l+n)^2(2+2l+n)}\nonumber 
\end{eqnarray}
\end{proof}

\subsubsection{Proof of Theorem \ref{Main2}}Let $(\Omega, ds^2, \mu)$ be a weighted bounded open  subset, with smooth boundary $\partial 
\Omega \neq \emptyset$, of a Riemannian weighted $m$-manifold $(M,ds^2, \mu)$.  The Green operator $G\colon L^2(\Omega, \mu)\to 
L^{2}(\Omega, \mu)$ is given by 
\begin{eqnarray}G(f)(x)&=&\int_{0}^{\infty}\int_{\Omega}p_t(x,y)f(y)d\mu(y)dt\nonumber \\
&=&\int_{\Omega} g(x,y)f(y)dy,\end{eqnarray}where $p_t(x,y)$ is the heat kernel of the operator $\mathcal{L}=-\triangle_{\mu}\vert_{W_{0}^{2}(\Omega, \mu)}$ and $g(x,y)$ is the Green function of $\Omega$. Let $\{u_1, u_2,u_3,  \ldots\} $ be a $L^{2}(\Omega, \mu)$-orthonormal basis of $L^2(\Omega, \mu)$ formed by eigenfunctions $u_i$ with eigenvalue $\lambda_i(\Omega)\in \sigma (\Omega)$.  The proof Theorem \ref{Main2} is divided in few simple propositions.
\begin{proposition} If $\lambda_{1}(\Omega)>0$ then the Green operator $G$ satisfies
\begin{equation}\label{eq43}
G(f)(x)=\sum_{i=1}^\infty\frac{f^i\,u_i(x)}{\lambda_i(\Omega)}\,\,{\rm for}\,\,{\rm any}\,\, f\in L^{2}(\Omega, \mu),
\end{equation} 
where $f^i=\int_\Omega f(x)u_i(x)d\mu(x)$. Moreover,
\begin{equation}\label{eq42}
G^k(f)=\sum_{i=1}^\infty\frac{f^i\,u_i(x)}{\lambda_i^k(\Omega)},
\end{equation}
where 
$G^{k}=\stackrel{k-\text{times}}{\overbrace{G\circ \cdots \circ G}}$.
\end{proposition}
\begin{proof}Let $\{u_i\}_{i=1}^{\infty}$ be a complete orthonormal basis of $L^2(\Omega)$ of eigenfunctions. For any $f\in L^2(\Omega, \mu)$ we have \[f(x)=\sum _{i=1}^{\infty}f^{i}u_i(x)\,\,{\rm in}\,\,L^{2}-{\rm sense}.\] This means that $\Vert f - \sum_{i=1}^{k} f^{i}u_i\Vert_{L^2}\to 0$ as $k \to \infty$. Set \[f_k(x)=\sum_{k+1}^{\infty}f^{i}u_i(x).\] Since $\sum_{i=k+1}^{\infty}\left(f^{i}\right)^2<\infty$ we have then $f_k\in L^{2}(\Omega, \mu)$. Thus $$U^k=f-f_k=\sum_{i=1}^{k}f^{i}u_i(x)\in L^{2}(\Omega, \mu).$$ Therefore, \begin{equation}\label{eqGimeno} \left\Vert f-\sum_{i=1}^{k}f^{i}u_i(x)\right\Vert_{L^{2}}=\Vert f-U^{k}\Vert=\Vert f_k\Vert_{L^{2}}\to 0\,\,{\rm as}\,\, k\to \infty.\end{equation}
For any $k\in \mathbb{N}$,
\begin{eqnarray}G(f)(x)&=& \int_{\Omega} g(x,y)f(y)d\mu(y)\nonumber \\
&=& \int_{\Omega}g(x,y)U^{k}(y)d\mu(y) +  \int_{\Omega}g(x,y)f_k(y)d\mu(y)\nonumber \\
&=& \int_{\Omega}g(x,y)\sum_{i=1}^{k}f^{i}u_i(y)d\mu(y)+  \int_{\Omega}g(x,y)f_k(y)d\mu(y)\label{eqGimeno2} \\
&=&\sum_{i=1}^{k}f^{i}\int_{\Omega}g(x,y)u_i(y)d\mu(y) + G(f_k)(x)\nonumber \\
&=&\sum_{i=1}^{k}\frac{f^{i}u_i(x)}{\lambda_i(\Omega)}+ G(f_k)(x).\nonumber
\end{eqnarray} Therefore, from \eqref{eqGimeno2} and \cite[Exercise 13.6]{grigoryan-book} we  have
\begin{equation}
\left\Vert G(f)-\sum_{i=1}^{k}\frac{f^{i}u_{i}}{\lambda_{i}(\Omega)}\right\Vert_{L^{2}}=\Vert G(f_k)\Vert_{L^{2}}
\leq \frac{\Vert f_k\Vert_{L^{2}}}{\lambda_{1}(\Omega)}\to 0 \,\, {\rm as}\,\, k\to \infty.
\end{equation}This proves \eqref{eq43}.
We used that $\int_{\Omega} g(x,y)u_i(y)d\mu (y)=G(u_{i})(x)=u_i(x)/\lambda_{i}(\Omega)$.
We will prove \eqref{eq42} by induction. Assume  that \[
G^k(f)(x)=\sum_{i=1}^\infty\frac{f^i\,u_i(x)}{\lambda_i^k(\Omega)}\in L^{2}(\Omega).
\]Then, since $(G^{k}(f))^i=\displaystyle\frac{f^i}{\lambda_i^k(\Omega)}$, we have \begin{eqnarray} G^{k+1}(f)(x)&=&G(G^{k}(f))(x)\nonumber \\
&=& \sum_{i=1}^\infty\frac{(G^{k}(f))^i\,u_i(x)}{\lambda_i(\Omega)}\\
&=&\sum_{i=1}^\infty\frac{f^i\,u_i(x)}{\lambda_i^{k+1}(\Omega)}\cdot\nonumber
\end{eqnarray}
\end{proof}

\begin{proposition}\label{propl1}Let $\Omega\subset M$ be a bounded open set  with smooth boundary $\partial \Omega$ such that $\lambda_{1}(\Omega)>0$. Then, for any $f\in L^2(\Omega)$,
\begin{equation}\label{eq45}
\Vert G^k(f) \Vert^2=\sum_{i=1}^\infty\frac{(f^i)^2}{\lambda_i^{2k}(\Omega)}\cdot
\end{equation}
Moreover,
\begin{equation}
\lim_{k\to\infty}\frac{\Vert G^k(f) \Vert}{\Vert G^{k+1}(f) \Vert}=\lambda_l(\Omega),
\end{equation}
\begin{equation}
\lim_{k\to \infty}\frac{G^k(f)}{\Vert G^{k}(f) \Vert}\to \phi_l\in {\rm Ker}(\triangle_{\mu}+\lambda_l)\,\,{\rm in }\,\, L^2,
\end{equation}
where $l$ is the first integer such that,
\begin{equation}
  \int_\Omega f(y)u_l(y)d\mu(y)\neq 0.
\end{equation}
\end{proposition}

\begin{proof}
Identity (\ref{eq45}) follows from equation (\ref{eq42}) and Passerval's identity.
\vspace{2mm}

 Using \eqref{eq42} we have
\begin{eqnarray}
\frac{\Vert G^k(f) \Vert^2}{\Vert G^{k+1}(f) \Vert^2}&=& \displaystyle\frac{\sum_{i=l}^\infty\displaystyle\frac{(f^i)^2}{\lambda_i^{2k}(\Omega)}}{\sum_{i=l}^\infty\displaystyle\frac{(f^i)^2}{\lambda_i^{2k+2}(\Omega)}}\nonumber \\ 
&&\\ &=&\lambda_l^2(\Omega)\displaystyle\frac{\sum_{i=l}^\infty(f^i)^2\left(\displaystyle\frac{\lambda_l(\Omega)}{\lambda_i(\Omega)}\right)^{2k}}{\sum_{i=l}^\infty(f^i)^2\left(\displaystyle\frac{\lambda_l(\Omega)}{\lambda_i(\Omega)}\right)^{2k+2}}\cdot\nonumber 
\end{eqnarray}
Since $\displaystyle\frac{\lambda_l(\Omega)}{\lambda_i(\Omega)}<1$ for any $i>l$, we have
\begin{equation}
\lim_{k\to\infty}\left(\displaystyle \frac{\lambda_l(\Omega)}{\lambda_i(\Omega)}\right)^{2k+2}=\lim_{k\to\infty}\left(\displaystyle\frac{\lambda_l(\Omega)}{\lambda_i(\Omega)}\right)^{2k}=\delta_{il}.
\end{equation}
Then,
\begin{equation}
\begin{aligned}
\lim_{k\to\infty}\frac{\Vert G^k(f) \Vert^2}{\Vert G^{k+1}(f) \Vert^2}=\lambda_l^2(\Omega).
\end{aligned}
\end{equation}This proves  identity \eqref{eq3.3} of Theorem \ref{Main2}.
By the identities (\ref{eq43}) and (\ref{eq45}) we have that
\begin{equation}
\begin{array}{lll}\displaystyle
\int_{\Omega}\displaystyle\frac{G^k(f)}{\Vert G^{k}(f) \Vert}u_id\mu &=&\displaystyle\frac{f^i}{\lambda_i^k(\Omega)}\frac{1}{\sqrt{\displaystyle\sum_{j=l}^\infty\displaystyle\frac{(f^j)^2}{\lambda_j^{2k}(\Omega)}}}\\
&&\\
&=&\left(\displaystyle\frac{\lambda_l(\Omega)}{\lambda_i(\Omega)}\right)^k\displaystyle\frac{f^i}{\sqrt{(f^l)^2+\displaystyle\sum_{j=l+1}^\infty(f^j)^2\left(\displaystyle\frac{\lambda_i(\Omega)}{\lambda_j(\Omega)}\right)^{2k}}}\cdot
\end{array}
\end{equation}
Since $\lambda_j(\Omega)<\lambda_i(\Omega)$, we have that
\begin{equation}
\lim_{k\to\infty}\int_{\Omega}\frac{G^k(f)}{\Vert G^{k}(f) \Vert}u_id\mu=\int_{\Omega}\lim_{k\to\infty}\frac{G^k(f)}{\Vert G^{k}(f) \Vert}u_id\mu=\delta_{li}.
\end{equation} This shows that  $\phi_l=\lim_{k\to\infty}\frac{G^k(f)}{\Vert G^{k}(f) \Vert}\in {\rm Ker}(\triangle_l+\lambda_l)$ and proves  identity \eqref{eq3.5} of Theorem \ref{Main2}.
\end{proof}

\begin{corollary}\label{cor17}
Under the assumptions of the above proposition and letting $f_1$ be the function 
\begin{equation}
f_1:=f-\lambda_l(\Omega)G(f),
\end{equation}
we have
\begin{equation}
\begin{array}{lll}
\lim_{k\to\infty}\displaystyle\frac{\Vert G^k(f_1)\Vert}{\Vert G^{k+1}(f_1)\Vert}&=&\lambda_n(\Omega)\\
&&\\
\lim_{k\to \infty}\displaystyle\frac{G^k(f_1)}{\Vert G^{k}(f_1)\Vert}&\to& \phi_n\in {\rm Ker}(\triangle_{\mu}+\lambda_n)\,\,{\rm in }\,\, L^2.
\end{array}
\end{equation}
with,
$$
\lambda_n(\Omega)>\lambda_l(\Omega).
$$
\end{corollary}
\begin{proof}
We only have to apply the above proposition taking into account that by equality (\ref{eq43})
\begin{equation}
\begin{array}{lll}
f-\lambda_l(\Omega)G(f)&=&\sum_{i=l}^\infty f^iu_i-\sum_{i=l}^\infty f^iu_i\displaystyle\frac{\lambda_l(\Omega)}{\lambda_i(\Omega)}\\
&&\\
&=&\sum_{i=l+1}^\infty f^iu_i\left(1-\displaystyle\frac{\lambda_l(\Omega)}{\lambda_i(\Omega)}\right).
\end{array}
\end{equation}
\end{proof}
Observe that since $u_1$ does not change its sign in $B_{\mathbb{M}_h}(r)$, 
$$
\int_{B_{\mathbb{M}_h}(r)} f(x) u_1(x) d\mu(x) \neq 0
$$ positive or negative) function $f$. Hence,  using Proposition \ref{propl1}, we have
\begin{corollary}\label{cor18}Let $\Omega\subset M$ be a bounded open set with smooth boundary $\partial \Omega\neq \emptyset$. Then, for any positive (or negative) $f\in L^2(\Omega, \mu)$,
\begin{equation}
\begin{array}{rll}
\lim_{k\to\infty}\displaystyle\frac{\Vert G^k(f) \Vert}{\Vert G^{k+1}(f) \Vert}&=&\lambda_1(\Omega),\\
&&\\
\lim_{k\to \infty}\displaystyle\frac{G^k(f)}{\Vert G^{k}(f) \Vert}&=& \phi_1 \in {\rm Ker}(\triangle_{\mu}+\lambda_1)\,\,{\rm in }\,\, L^2.
\end{array}
\end{equation}
\end{corollary}

\subsubsection{Proof of Theorem \ref{thm4.2} }

Let $\Omega$ be an open  relatively compact domain with smooth boundary $\partial \Omega\neq \emptyset $ of a Riemannian manifold $M$. Consider the following hierarchy  Dirichlet problem 
\begin{equation}\label{eq58}
\left\{\begin{array}{rll}
\phi_0&=&1\,\, \text{ on }\,\,\Omega\\
\triangle \phi_k+k\phi_{k-1}&=&0\,\,\text{ on }\,\,\Omega\\
{\phi_k}_{\vert \partial\Omega}&=&0.
\end{array}\right.
\end{equation}Theorem \ref{thm4.2} states that this problem 
admits a unique family of solutions $\{\phi_k\}_{k=1}^\infty$, given by
\begin{equation}\label{eq59}
\phi_k(x)=k!\, G^k(1)(x).
\end{equation}
In order to show the uniqueness, note that first of all that $\phi_1$ is unique. Otherwise we would have  two different functions $\phi_1^1$ and $\phi_1^2$ such that
\begin{equation}
\begin{aligned}
&\triangle \phi_1^1=\triangle \phi_1^2=-1\\
&\phi_1^1\vert_{\partial \Omega}=\phi_1^2\vert_{\partial \Omega}=0.
\end{aligned}
\end{equation}
Hence $\phi_1^1-\phi_1^2$ would  be an harmonic function in $\Omega$ with $\phi_1^1-\phi_1^2=0$ in the boundary $\partial \Omega$, then by the minimum principle $\phi_1^1=\phi_1^2$,  a contradiction. Assume that in the family of solutions $\{ \phi_{k}\}_{k=1}^{\infty}$ the first $j$ functions $\phi_{1},\phi_{2},\ldots, \phi_{j}$ are unique but there  exists two different  $\phi_{j+1}^1$ and $\phi_{j+1}^2$ solutions in the $j+1$-th slot. Then $\triangle (\phi_{j+1}^1-\phi_{j+1}^2)=0$ and $\phi_{j+1}^1-\phi_{j+1}^2=0$ on $\partial \Omega$. Then $\phi_{j+1}^1=\phi_{j+1}^2$ 
Now, functions defined by
$\phi_k(x)=k\, G(\phi_{k-1})(x)$, $k=1,\ldots$
are the solutions of the problem \eqref{eq58}.
Since the solution to the problem (\ref{eq58}) is unique we only have to check that 
\begin{equation}
\triangle \left(k\, G(\phi_{k-1})\right)=-k\phi_{k-1}.
\end{equation}
But that is straightforward because  the Green operator is the inverse of $-\triangle$. To prove  (\ref{eq59}) let us use the induction method. Observe that $\phi_1=G(1)(x)$. Suppose that equation (\ref{eq59}) is true and let us compute $\phi_{k+1}(x)$.
\begin{equation}
\begin{aligned}
\phi_{k+1}(x)=&(k+1)G(\phi_k)(x)=(k+1)G(k!\,G^k(1))(x)\\
=&(k+1)k!\,G(G^k(1))(x)=(k+1)!\,G^{k+1}(1)(x).
\end{aligned}
\end{equation}
This proves \eqref{eqphi_k} in Theorem \ref{thm4.2}. On the other hand, by \eqref{eq42},  \[G^{k}(1)(x)=\sum_{i=1}^{\infty}\frac{a_i u_i(x)}{\lambda_{i}^{k}(\Omega)},\] where $a_i=\int_{\Omega}u_i(y)d\mu(y)$. Thus \[\phi_{k}(x)=k!G^{k}(1)=k!\sum_{i=1}^{\infty}\frac{a_i u_i(x)}{\lambda_{i}^{k}(\Omega)}.\]Moreover, the $L^{1}(\Omega, \mu)$-momentum spectrum of $\Omega$ is readily obtained by \[\mathcal{A}_{k}(\Omega)=\int_{\Omega}\phi_{k}d\mu=k!\sum_{i=1}^{\infty}\frac{ a_{i}^{2} }{\lambda_{i}^{k}(\Omega)}.\]
This proves Theorem \ref{thm4.2}. The proof of Corollary \ref{cor4.1} is straightforward.

\end{document}